\documentclass[11pt, a4paper, final]{article}

\usepackage[width=15.92cm]{geometry}

\usepackage{amsfonts}
\usepackage{amsmath}
\usepackage{amssymb}
\usepackage{amsthm}
\usepackage{bm}
\usepackage{cite}
\usepackage{showkeys}
\usepackage[english]{babel}
\usepackage{dblaccnt}
\usepackage{accents}
\usepackage{graphicx}
\usepackage{psfrag}
\usepackage{subfig}
\usepackage{color}
\usepackage{amscd}
\usepackage[mathscr]{euscript}
\usepackage{showkeys}
\usepackage{hyperref}

\newtheorem{defi}{Definition}[section]
\newtheorem{lemma}[defi]{Lemma}
\newtheorem{theorem}[defi]{Theorem}

\newtheorem{proposition}[defi]{Proposition}

\newtheorem{remark}[defi]{Remark}

\renewcommand{\i}{\mathrm{i}}
\renewcommand{\d}[1]{\,\mathrm{d}#1 \,}

\newcommand{\ol}[1]{\overline{#1}}
\newcommand{\K}{\mathcal{K}}
\newcommand{\epsr}{\epsilon_{\mathrm{r}}}
\renewcommand{\epsilon}{\varepsilon}
\newcommand{\loc}{\mathrm{loc}}
\newcommand{\per}{\mathrm{per}}
\newcommand{\smooth}{\text{sm}}
\newcommand{\RR}{R}
\newcommand{\T}{\mathcal{T}}
\newcommand{\F}{\mathcal{F}}

\DeclareMathOperator{\supp}{\mathrm{supp}}

\renewcommand{\Re}{\mathrm{Re}\,}
\renewcommand{\Im}{\mathrm{Im}\,}

\DeclareMathOperator{\curl}{curl}

\renewcommand{\div}{\mathrm{div} \,}

\newcommand*{\N}{\ensuremath{\mathbb{N}}}
\newcommand*{\Z}{\ensuremath{\mathbb{Z}}}
\newcommand*{\R}{\ensuremath{\mathbb{R}}}
\newcommand*{\C}{\ensuremath{\mathbb{C}}}
\begin{document}
\sloppy

\title{A Trigonometric Galerkin Method for Volume Integral Equations Arising in TM Grating Scattering}
\author{Armin Lechleiter\thanks{Center for Industrial Mathematics, University of Bremen, 28359 Bremen, Germany}
\and Dinh-Liem Nguyen\thanks{DEFI, INRIA Saclay--Ile-de-France and Ecole Polytechnique, 91128 Palaiseau, France}} 

\maketitle

\begin{abstract}
  Transverse magnetic (TM) scattering of an electromagnetic wave 
  from a periodic dielectric diffraction grating 
  can mathematically be described by a volume integral equation. 
  This volume integral equation, however, in general  
  fails to feature a weakly singular integral operator.
  Nevertheless, after a suitable periodization, the involved integral operator can be 
  efficiently evaluated on trigonometric polynomials using the 
  fast Fourier transform (FFT) and iterative methods can be used to solve 
  the integral equation. Using Fredholm theory, we prove that a trigonometric 
  Galerkin discretization applied to the periodized integral equation converges 
  with optimal order to the solution of the scattering problem. 
  The main advantage of this FFT-based discretization scheme is that the 
  resulting numerical method is particularly easy to implement, avoiding for 
  instance the need to evaluate quasiperiodic Green's functions.  
\end{abstract}

\section{Introduction}

Periodic dielectric structures are important ingredients for 
modern optical technologies, serving as beam splitters, lenses, 
monochromators, and spectrometers. 
Simulation of electromagnetic fields in such periodic structures is a 
challenging task, since the wave field oscillates in an 
unbounded domain, since the quasi-periodicity needs to be taken into 
account, and since evanescent waves arise around the structure. 
Hence, it might be difficult to use, e.g., a 
standard finite element software for the simulation of wave fields
in such structures. For this reason, this paper presents a simple-to-implement 
volume integral equation solver for this simulation task.

We consider scattering of time-harmonic electromagnetic 
waves from diffraction gratings, three dimensional dielectrics 
that are periodic in one spatial direction and invariant in a second, 
orthogonal, direction (compare Figure~\ref{fig:0}). If the incident wave 
is a transverse-magnetic (TM) wave, the electromagnetic field 
can be described by the scalar equation 
\begin{equation}
  \label{eq:basic}
  \div ( a \nabla u) + k^2  u = 0,
\end{equation}
with wave number $k>0$, see, e.g.,~\cite{Nedel2001}. 
The real material parameter $a$ is in this paper assumed to be scalar, positive 
and possibly discontinuous. This periodic scattering problem can be equivalently 
reformulated as a volume integral equation that is formally of the second kind. 
However, since the coefficient $a$ in~\eqref{eq:basic} appears in the highest-order 
term, the integral operator of this volume integral equation fails to be compact unless 
$a$ is not globally smooth (compare, for the case of Maxwell's equations,~\cite[Chapter 9.2]{Colto1992a}). 
The aim of this paper is to analyze the convergence of a trigonometric Galerkin 
discretization of this volume integral equation for discontinuous material parameter $a$. 
This analysis will be partly based on (purely analytic) results from the 
paper~\cite{Lechl2012a}. Here, we adapt the volume integral equations corresponding 
to~\eqref{eq:basic} such that they can be numerically treated via an FFT-based approach. 
This resulting numerical scheme can be rigorously shown to be (quasi-optimally) convergent.
We provide fully discrete formulas for the implementation of the scheme together with computational examples. 

\begin{figure}[h!!!!tb]
  \begin{center}
    \psfrag{x1}{$x_1$}
    \psfrag{x2}{$x_2$}
    \psfrag{x3}{\hspace*{-3mm}$x_3$}
    \includegraphics[width=7cm]{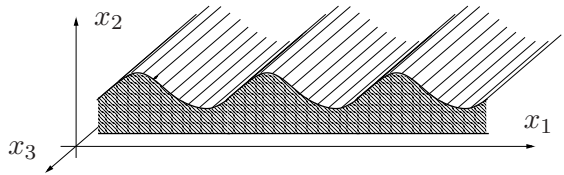}
    \caption{Sketch of the diffraction grating under consideration.}
  \label{fig:0}
  \end{center}
\end{figure}

It might seem inappropriate to consider the TM mode equation~\eqref{eq:basic}, since 
the corresponding transverse electric (TE) mode yields the well-known Lippmann-Schwinger 
integral equation that features a weakly singular integral operator. Indeed, the numerical 
scheme from~\cite{Vaini2000} for the Lippmann-Schwinger equation inspired the scheme 
we develop here. However, if materials feature both dielectric and magnetic 
contrast then highest-order coefficients cannot be avoided even in the TE or TM mode 
problems. Note that it would not be too difficult to construct numerical schemes for the 
simulation of such materials by combining the one from this paper with, e.g., schemes 
developed earlier for the Lippmann-Schwinger equation.

Volume integral equations are a standard numerical tool in 
the engineering community to solve scattering problems numerically, see, 
e.g.,~\cite{Richm1965, Richm1966, Zwamb1992, Kottm2000, Ewe2007}. 
The linear system resulting from the 
discretization of the integral operator (usually done by collocation  
or finite element methods) is large and dense. Fortunately, the convolution 
structure of the integral operator allows to compute matrix-vector 
multiplications by the FFT in an order-optimal way 
(up to logarithmic terms), see, e.g.,~\cite{Zwamb1992, Rahol1996}, at
least if the discretization respects this convolution structure. 
This partly explains the success of such methods in applications.
However, the discretization of the integral operator itself is at least in 
some works done in a mathematically crude way and a rigorous convergence 
analysis for the different discretization techniques is usually missing. 

Despite their relevance in applications, volume integral equations 
featuring \emph{strongly singular}  integral operators (i.e., integral operators that 
fail to be weakly singular) are a recent analytic research subject in mathematics, 
see, e.g.,~\cite{Potth1999, Kirsc2009, Costa2010, Costa2011, Lechl2012a}. 
In particular, the numerical analysis of practically feasible discretization 
methods based on these equations seems to be in a somewhat premature stage.
Of course, one reason for this phenomenon is that for many relevant 
material configurations, the need for discretizing a strongly singular 
volume integral equation can be avoided. For example, whenever material 
parameters are piecewise constant, boundary integral equations  
are a powerful alternative to the volumetric approach, see, e.g.,~\cite{Otani2009}
for a recent reference dealing with a periodic scattering problem. 
If the material parameters fail to be piecewise constant, an important approach 
to avoid the discretization of strongly singular  
integral operators is to combine volume and surface integral operators. 
For the full Maxwell's equations in free space, the analytic equivalence of both 
the volume integral equation and the coupled system of weakly singular volume 
and surface integral operators has been worked out in detail in~\cite{Costa2010}. 

However, whenever using (possibly coupled) boundary integral equations 
one usually needs to be able to rapidly and accurately evaluate the 
underlying Green's function. It is well-known that this is a non-trivial task 
for (quasi-)periodic Green's functions, see, e.g.~\cite{Linto1998}, 
becoming even more difficult if additionally multi-pole expansions are 
used as in~\cite{Otani2009}. The numerical scheme presented here does not require 
to evaluate Green's functions and it is in principle applicable to arbitrary 
varying material parameters. However, the scheme explicitly requires the 
(two-dimensional) Fourier coefficients of the material parameter. According to 
our experience, the accuracy of computational results improves considerably 
if these coefficients can be computed analytically, or at least be reduced to some 
semi-analytic form that can easily be treated numerically with high accuracy. The latter 
is for instance the case for piecewise polynomial or trigonometric material parameters, 
as we illustrate through examples in the last section. 

Our numerical analysis of a trigonometric Galerkin discretization  
applied to the volume integral equation relies in parts on 
G\r{a}rding inequalities that we proved in~\cite{Lechl2012a}. 
Of course, these inequalities would in principle directly 
justify any Galerkin discretization of the integral equation. 
However, such a discretization does generally not profit from 
the above-described advantages arising from the convolution 
structure of the integral operator, the related diagonalization 
of the operator on trigonometric polynomials, and the possibility 
of rapidly evaluating the integral operator using the FFT. 
Additionally, when discretizing the integral operator using finite 
elements, the strong singularity of the kernel makes the computation 
of the diagonal of the system matrix challenging, see~\cite{Kone2010}. 
To this end, we first periodize the integral operator before discretizing, 
using a technique that was (up to a smoothing procedure) analogously used in~\cite{Vaini2000}.
The periodized operator is then easily evaluated spectrally, 
since one can (almost) explicitly compute its Fourier coefficients
(see~\eqref{eq:evaluateL}). Due to the lack of compactness of the integral 
operator it seems difficult to analyze collocation discretizations as it 
was originally done in~\cite{Vaini2000}. However, it is still 
possible to fully analyze a Galerkin discretization 
(see Proposition~\ref{th:convergence}). 

In essence, the advantage of this trigonometric Galerkin discretization  
is that it is particularly simple to implement -- the core of our implementation 
takes less than 70 lines in MATLAB -- and that the linear system can 
be evaluated at FFT speed. By using relatively simple parallelization 
techniques on modern multi-core processors this allows to evaluate 
the integral operator rapidly (MATLAB even automatically uses 
parallelized FFTW routines~\cite{Frigo2005}). Additionally, the FFT-based 
method requires no evaluation of the quasiperiodic Green's function 
or of its partial derivatives. Due to the slow convergence of standard 
expressions of this Green's function, sophisticated techniques like 
Ewald summation need to be used to accurately evaluate them. 
Of course, the price to pay for these advantages is that the convergence 
order of this FFT-based method is low if the medium has jumps, 
due to the use of global trigonometric basis functions 
(otherwise the method is high-order convergent). Nevertheless, if one is
merely interested in obtaining a moderately accurate solution 
without investing much implementation work, we are convinced 
that the method presented here is an interesting simulation technique. 
This technique could be further improved by using non-uniform FFTs that 
allow some refinement of the underlying grid of the FFT close to edges of 
the structure, for instance. See, e.g.,~\cite{Nie2005, Zhang2002} for references 
on non-uniform FFTs and their use to solve volume integral equations.

%
%

The rest of this paper is organized as follows: In Section~\ref{se:direct} we briefly 
recall the volumetric integral equation for the direct scattering problem 
and the corresponding G\r{a}rding inequality from~\cite{Lechl2012a}. 
In Section~\ref{se:periodize} we periodize the volume integral equation 
such that it is suitable for a fast FFT-based discretization on biperiodic trigonometric 
polynomials. We also prove the necessary G\r{a}rding inequalities for the periodized system
(see Theorem~\ref{th:PeriodicGarding}). These inequalities are naturally the 
basis for quasi-optimal error estimates for the trigonometric Galerkin 
discretization in Section~\ref{se:discretization}. Finally, Section~\ref{se:numerics}
contains several illustrative numerical examples.

\emph{Notation:} $L^2$-based Sobolev spaces on a domain 
$D$ are denoted as $H^s(D)$, $s \in \R$, and 
$C^{m,1}(\ol{D})$ is the usual space of Lipschitz continuous 
functions that possess Lipschitz continuous partial derivatives 
up to order $m$. 
Further, $H^s_{\loc}(D) = \{ v\in H^s(B) \text{ for all open balls } B \subset D \}$.
The trace of a function $u$ on the boundary $\partial D$ from the 
outside and from the inside of $D$ is denoted as $\gamma_{\mathrm{ext}}(u)$ and 
$\gamma_{\mathrm{int}}(u)$, respectively. The jump of $u$ across $\partial D$ is 
$[u]_{\partial D} = \gamma_{\mathrm{ext}}(u) - \gamma_{\mathrm{int}}(u)$.
If the exterior and the interior trace of a function $u$
coincide, then we simply write $\gamma(u)$ for the trace.

\section{Problem Setting and Known Results}
\label{se:direct}

Propagation of time-harmonic electromagnetic waves in an inhomogeneous,
isotropic, and lossless medium is described by the Maxwell's equations for the 
electric and magnetic fields $E$ and $H$, respectively, 
$ \curl H + \i\omega\varepsilon E = 0$ and $\curl E - \i\omega\mu_0 H = 0$ in $\R^3$.
Here, $\omega>0$ denotes the frequency, $\varepsilon$ is the positive electric 
permittivity and $\mu_0$ is the (constant and positive) magnetic permeability. 
We assume in this paper that the scalar function $\epsilon$ is independent of 
the third variable $x_3$, and $2\pi$-periodic in the first variable $x_1$. Further, 
we suppose that $\epsilon$ equals a constant $\epsilon_0>0$ outside the 
grating structure. 

If an incident electromagnetic plane wave independent of 
the third variable $x_3$ illuminates the grating, then the
Maxwell's equations for the total wave field decouple into 
two scalar partial differential equations. In particular, the 
third component $H_3$ of the magnetic field satisfies 
\begin{equation}
 \label{eq:HMode}
 \div \left( \epsr^{-1} \nabla u\right) + k^2 u = 0
 \qquad \text{with }
 \epsr := \epsilon / \epsilon_0
 \text{ and } k := \omega\sqrt{\epsilon_0 \mu_0} > 0 ,
\end{equation}
together with jump conditions on interfaces where the 
refractive index $\epsr^{-1}$ jumps: $u$ and 
$\epsr^{-1} \, \partial u / \partial \nu$ are 
continuous across such interfaces. Note that 
$\epsilon_r$ is $2\pi$-periodic in $x_1$. 
We assume that the contrast $q := \epsr^{-1} - 1$ 
has support in $\{ |x_2| < \rho \}$ for some constant 
$\rho>0$. 

Consider now a plane incident wave  $u^i(x) = \exp(\i k \, x \cdot d)
= \exp(\i k(x_1d_1 + x_2d_2))$ where $|d|=1$ and $d_2 \not = 0$.
When $u^i$ illuminates the diffraction 
grating there arises a scattered field $u^s$ such that the 
total field $u=u^i+u^s$ satisfies~\eqref{eq:HMode}, that is, 
the scattered field satisfies
\begin{equation}
 \label{eq:HmodeEquation}
 \div( \epsr^{-1} \nabla u^s) + k^2 u^s = -\div(q \nabla u^i)
 \quad \text{in } \R^2.
\end{equation}
Note that $u^i$ is $\alpha$-quasi-periodic with respect to $x_1$, 
\[
 u^i(x_1 + 2\pi,x_2) = e^{2\pi \i \alpha} u^i(x_1,x_2)
 \qquad \text{for $\alpha: = kd_1$.}
\]
Since $\epsr$ is periodic we seek for a scattered field that is 
$\alpha$-quasi-periodic in $x_1$, too. For uniqueness of solution  
we require that $u^s$ above (below) the dielectric structure can be
represented by a uniformly converging Rayleigh
series consisting of upwards (downwards) propagating or
evanescent plane waves,
\begin{equation}
  \label{eq:RayleighCondition}
  u^s(x) = \sum_{j \in \Z} \hat{u}^\pm_j
  e^{\i\alpha_jx_1 \pm \i\beta_j (x_2\mp\rho)}, \quad x_2 \gtrless \pm\rho,
  \qquad \alpha_j := j + \alpha, \quad \beta_j := \sqrt{k^2-\alpha^2_j}.
\end{equation}
The square root used to define 
\[
  \beta_j = \sqrt{k^2-\alpha^2_j} := 
  \begin{cases} 
    (k^2-\alpha^2_j)^{1/2}, & k^2 \geq \alpha_j^2, \\ 
    \i (\alpha^2_j - k^2)^{1/2}, & k^2 < \alpha_j^2,
  \end{cases}, \qquad 
  j \in \Z,
\]
is chosen such that $\Im (\beta_j) \geq 0$ always. 
Further, the so-called Rayleigh coefficients $\hat{u}^{\pm}_j$ 
of the scattered wave in~\eqref{eq:RayleighCondition} 
have explicit representations, 
\[
  \hat{u}^{\pm}_j = \frac{1}{2\pi}
    \int_{-\pi}^{\pi} u^s(x_1,\pm \rho) e^{-\i\alpha_j x_1} \d{x_1}, \qquad j \in \Z.
\]
Note that we call a solution to the Helmholtz equation \emph{radiating}   
if it satisfies~\eqref{eq:RayleighCondition}. 

By $G_{\alpha}$ we denote the Green's function to the
$\alpha$-quasi-periodic Helmholtz equation in $\R^2$, see, e.g.,~\cite[Eq.~(2.13)]{Linto1998}. 
In this paper,
\begin{equation}
  \label{eq:nonResonance}
  \text{ we always suppose that } \quad 
  k^2 \neq \alpha^2_j \qquad \text{for all } j \in \Z,
\end{equation}
which implies that this Green's function has the series representation 
\begin{equation}
  \label{eq:GkAlpha}
  G_{\alpha}(x)
  := \frac{\i}{4\pi} \sum_{j \in \Z}
  \frac{1}{\beta_j} \exp(\i\alpha_j x_1 + \i\beta_j |x_2|)
  \quad \text{for }  x= \left( \begin{matrix} x_1 \\ x_2 \end{matrix} \right), \ 
  x \not = \left( \begin{matrix} 2\pi m \\ 0 \end{matrix} \right), \ m\in\Z.
\end{equation}
Note that~\eqref{eq:nonResonance} implies that all the 
$\beta_j=(k^2-\alpha^2_j)^{1/2}$ are non-zero, and that the 
Green's function is well-defined, see again~\cite{Linto1998}.  

\begin{remark}[Failure at Wood's anomalies]
  \label{eq:wood}
  The phenomenon that condition~\eqref{eq:nonResonance} fails to hold for some $k>0$ 
  is called a Wood's anomaly, see, e.g.,~\cite{Barne2011}. At a Wood's anomaly, the 
  representation~\eqref{eq:GkAlpha} is obviously not well-defined. Image-like 
  representations of the Green's function would (at least formally) be well-defined, see, 
  e.g.,~\cite[Eq.~(2.7)]{Linto1998} for an example. Hence, it might seem as if there was 
  a chance that the method presented in this paper works at Wood's anomalies.
  However, by carefully checking Lemma~\ref{th:boundKh} below one notes that this is 
  not the case since certain Fourier coefficients (denoted by $\hat\K_\rho$ later on) 
  are not well-defined at Wood's anomalies. 
\end{remark}

We introduce the strip $\Omega := (-\pi,\pi)\times\R$ and set 
\[
  \Omega_\RR := (-\pi,\pi) \times (-\RR, \RR) \quad \text{for } \RR>0.
\]
Moreover, we set  
$H^\ell_{\alpha}(\Omega_\RR) := \{u\in H^\ell(\Omega_\RR): \,  u=U|_{\Omega_\RR}
  \text{ for some } \alpha\text{-quasi-periodic } U \in H^\ell_{\loc}(\R^2) \}$ 
  for $\ell \in \N$ and $R>0$, and $H^1_\alpha(\Omega)$  
is defined analogously. For any Lipschitz domain $D$ (see~\cite{McLea2000} 
for a definition), the space $L^2(D, \C^2)$ contains all square integrable functions 
with values in $\C^2$ (complex column vectors with two components).  

\begin{lemma}[Lemmas 5 and 6 in~\cite{Lechl2012a}]
  \label{th:H2smoothness}
  If $D \subset \Omega$ is a Lipschitz domain, then the volume potential
  \[
    (V f)(x) = \int_{D} G_{\alpha}(x-y) f(y) \d{y},
    \quad x \in \Omega_\RR,
  \]
  is bounded from $L^2(D)$ into $H^2_\alpha(\Omega_\RR)$
  for all $\RR>0$.
  For $g \in L^2(D, \C^2)$ the potential
  $w = \div V g$ belongs to $H^1_\alpha(\Omega_\RR)$
  for all $\RR>0$. It is the unique
  radiating weak solution to $\Delta w + k^2 w = -\div g$
  in $\Omega$, that is, it satisfies the Rayleigh
  expansion condition~\eqref{eq:RayleighCondition}, and
  \begin{equation}
    \int_{\Omega} (\nabla w\cdot\nabla \ol{v} - k^2 w\ol{v})\d{x}
    = -\int_{D} g\cdot\nabla\ol{v} \d{x} 
    \quad 
    \text{for all $v \in H^1_\alpha(\Omega)$ with compact support.}
  \end{equation}
\end{lemma}

Let us now come back to the differential equation~\eqref{eq:HmodeEquation}
for the scattered field $u^s$. Recall that we assumed that the contrast 
$q = \epsr^{-1}-1$ has support in $\{ |x_2| < \rho \}$ for some $\rho>0$.
We denote this support (restricted to one period $- \pi < x_1 < \pi$) by 
\[
   \ol{D} = \supp(q)
\]
and suppose from now on that $D$ is a Lipschitz domain. 
If we choose $\RR>\rho$, then $\ol{D} \subset \Omega_\RR$. 
Moreover, by setting $f=q\nabla u^i$ in~\eqref{eq:HmodeEquation} 
the variational formulation of~\eqref{eq:HmodeEquation} reads 
\begin{equation}
  \label{eq:variationalForm}
  \int_{\Omega} (\nabla u^s \cdot \nabla\ol{v} - k^2 u^s \ol{v}) \d{x}
  = -\int_{D} (q \nabla u^s + f) \cdot \nabla \ol{v} \d{x}
\end{equation}
for all $v \in H^1_\alpha(\Omega)$ with compact
support in $\ol{\Omega}$. 
From Lemma~\ref{th:H2smoothness} we know that
the radiating solution to this problem is given by
$u^s = \div V(q\nabla u^s + f)$. 
If we define the bounded linear operator 
\[
  L: \, L^2(D, \C^2) \to H^1_\alpha(D), \quad f \mapsto \div V f,
\] 
then the scattered field $u^s$, solution to~\eqref{eq:HmodeEquation}, 
hence solves the volume integral equation
\begin{equation}
  \label{eq:lippmannD}
  u^s -  L(q \nabla u^s) = L(f) \quad \text{in } H^1_\alpha(D)
\end{equation}
for $f=q\nabla u^i$. The operator on the left of the last equation 
satisfies a G\r{a}rding inequality. 

\begin{theorem}[Theorem 16 in~\cite{Lechl2012a}]
\label{th:FreeGarding}
Assume that $q \geq q_0 >0$ in $D$,  that 
 $\sqrt{q} \in C^{2,1}(\ol{D})$, and that 
$D$ is of class $C^{2,1}$. There exists a compact operator $K$ on $H^1_\alpha(D)$ such that
\begin{equation*}
  \Re\langle v-L(q\nabla v),v\rangle_{H^1_\alpha(D)} \geq \|v\|^2_{H^1_\alpha(D)}
  - \Re\langle K v, \, v\rangle_{H^1_\alpha(D)}, \qquad v \in H^1_\alpha(D).
\end{equation*}
\end{theorem}

For a real-valued contrast, uniqueness of solution of the scattering
problem~(\ref{eq:HmodeEquation}-\ref{eq:RayleighCondition}), 
or equivalently of the integral equation~\eqref{eq:lippmannD},
does in general only hold for all but a discrete set of 
positive wave numbers. Uniqueness results either require 
(partially) absorbing materials or non-trapping coniditions on the 
material; examples of such conditions are given, e.g.,  
in~\cite{Bonne1994, Elsch1998}. 

\begin{remark}[Assumption on uniqueness of solution]
In the rest of the paper, 
we always suppose that uniqueness of solution to 
(\ref{eq:HmodeEquation}-\ref{eq:RayleighCondition}) holds.
\end{remark}
 
We restrict our theoretical analysis to real and positive contrasts, 
since G\r{a}rding inequalities corresponding to complex-valued or 
negative contrasts are more involved, see~\cite{Lechl2012a}.
Treating these cases would increase technicalities without adding 
new ideas to the text. 

\section{Periodization of the Integral Equation}
\label{se:periodize}

In this section we periodize the volume integral equation~\eqref{eq:lippmannD} 
and show the equivalence of the periodized equation and the original one. 
The purpose of this periodization is that the resulting
integral operator is, roughly speaking, diagonalized
by trigonometric polynomials. This allows to use fast FFT-based
schemes to discretize the periodized operator and 
iterative schemes to solve the discrete system. 
We also prove G\r{a}rding inequalities for the 
periodized integral equation, which turns out to be 
involved. However, these estimates are crucial to establish 
convergence of the discrete schemes later on. 

Let us again emphasize that we assume in all the paper that 
the non-resonance condition~\eqref{eq:nonResonance} is satisfied, 
which excludes Wood's anomalies.

Since we are interested in spectral schemes 
we define a periodized Green's function, firstly setting
\begin{equation}
  \label{eq:DefineKernel}
  \K_\rho(x) := G_{\alpha}(x),
  \quad x=(x_1,x_2)^\top \in \R \times (-\rho, \rho), \
  x \not = (2\pi m, 0)^\top \text{ for } m\in\Z,
\end{equation}
and secondly extending $\K_\rho(x)$ $2\rho$-periodically
in $x_2$ to $\R^2$. The trigonometric polynomials
\begin{equation}
  \label{eq:trigonometricPolynomial}
  \varphi_j(x)
  := \frac{1}{\sqrt{4\pi\rho}}
  \exp\Big( {\i(j_1 + \alpha)x_1 + \i\frac{j_2\pi}{\rho}x_2}\Big),
  \quad j = (j_1,j_2)^\top \in \Z^2,
\end{equation}
are orthonormal in $L^2(\Omega_{\rho})$. They differ 
from the usual Fourier basis only by 
a phase factor $\exp(\i \alpha x_1)$, and hence also 
form a basis of $L^2(\Omega_{\rho})$. For 
$f\in L^2(\Omega_{\rho})$ and $j =(j_1,j_2)^\top \in \Z^2$,
$\hat{f}(j) := \int_{\Omega_{\rho}} f \, \ol{\varphi_j} \d{x}$
are the Fourier coefficients of $f$.
For $0\leq s<\infty$ we define a fractional Sobolev space 
$H^s_{\per}(\Omega_{\rho})$ as the subspace of
functions in $L^2(\Omega_{\rho})$ such that
\[
  \| f \|^2_{H^s_{\per}(\Omega_{\rho})}
  := \sum_{j \in \Z^2}(1 + |j|^2)^s |\hat{f}(j)|^2 < \infty.
\]
It is well-known that for integer values of $s$, these 
spaces correspond to spaces of $\alpha$-quasi-periodic functions that are
$s$ times weakly differentiable, and that the above norm is then
equivalent to the usual integral norms.


\begin{lemma}[Theorem 2 in~\cite{Lechl2012a}]
  \label{th:boundKh}
  The Fourier coefficients of the kernel $\K_\rho$
  from~\eqref{eq:DefineKernel} are given by
  \[
    \hat\K_\rho(j)
    = \begin{cases}
      \frac{1}{\sqrt{4\pi \rho}} \frac{\cos(j_2 \pi) \exp(\i \beta_{j_1} \rho)-1}
        { k^2 -(j_1 + \alpha)^2 - ( j_2 \pi / \rho)^{2}} &  \text{for }  k^2 \neq (j_1 + \alpha)^2 - \big( \frac{j_2 \pi}{\rho} \big)^{2}, \\[1mm]
      \frac{\i}{4 j_2} \left(\frac{\rho}{\pi}\right)^{3/2} & \text{else},
    \end{cases}
    \qquad j= \left( \begin{matrix} j_1 \\ j_2 \end{matrix} \right) \in \Z^2.
  \]
  The convolution operator $K_\rho$, defined by 
  $(K_\rho f)(x) = \int_{\Omega_{\rho}} \K_\rho(x-y) f(y) \d{y}$
  for $x\in \Omega_{\rho}$, is bounded from 
  $L^2(\Omega_{\rho})$ into $H^2_{\per}(\Omega_{\rho})$.
\end{lemma}

The periodized kernel $\K_\rho$ from~\eqref{eq:DefineKernel} 
is not smooth at the boundaries 
$\{ x_2 = \pm \rho \}$. To prove G\r{a}rding inequalities for the
periodized integral equation, we additionally need to smoothen 
this kernel at $\{ x_2 = \pm \rho \}$ and, to this end, introduce a 
suitable cut-off function. For $\RR > 2\rho$ we choose a 
$2\RR$-periodic function $\chi \in C^3(\R)$ that satisfies
$0\leq \chi \leq 1$ and $\chi(x_2) = 1$ for
$|x_2| \leq 2\rho$. Moreover, we assume that $\chi(\RR)$ 
vanishes up to order three, $\chi^{(j)}(\RR)=0$
for $j=1,2,3$ (compare Figure~\ref{fig:2}).

\begin{figure}[b!!!th]
  \psfrag{x1}{\small $x_1$}
  \psfrag{x2}{\small $x_2$}
  \psfrag{rho}{\small $x_2=\rho$}
  \psfrag{-rho}{\small $x_2=-\rho$}
  \psfrag{R}{\small $x_2=R$}
  \psfrag{-R}{\small $x_2=-R$}
  \centering
\subfloat[]{\hspace*{-0.7cm}\includegraphics[width = 0.50\linewidth]{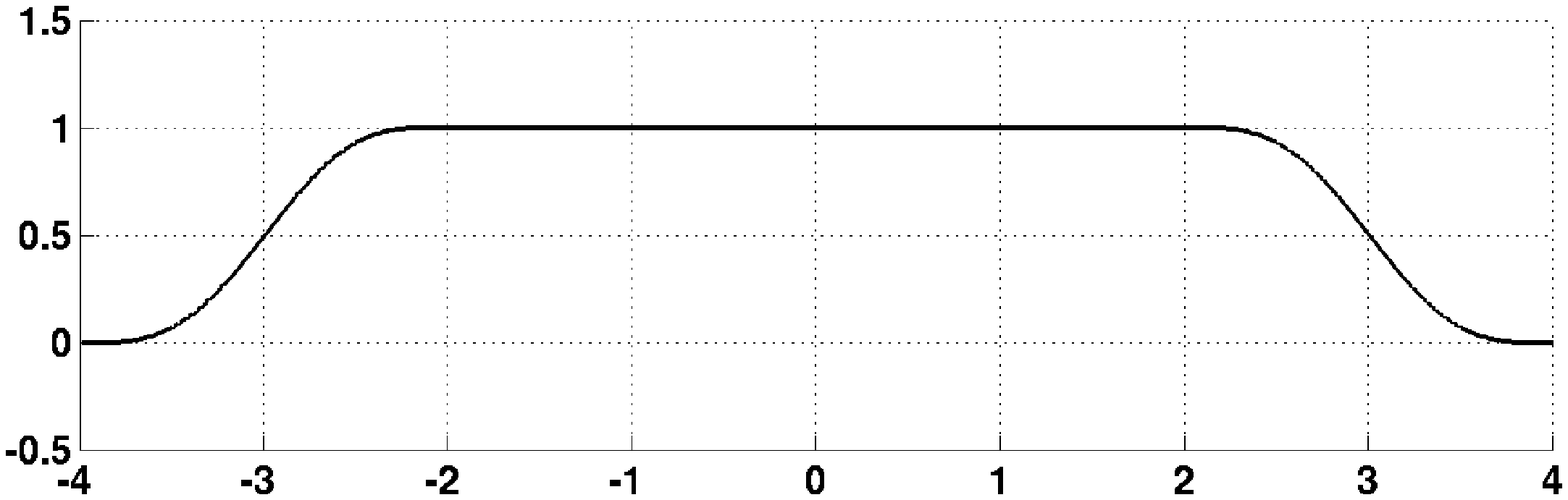}}
\subfloat[]{\hspace*{-0.3cm}\includegraphics[width = 0.55\linewidth]{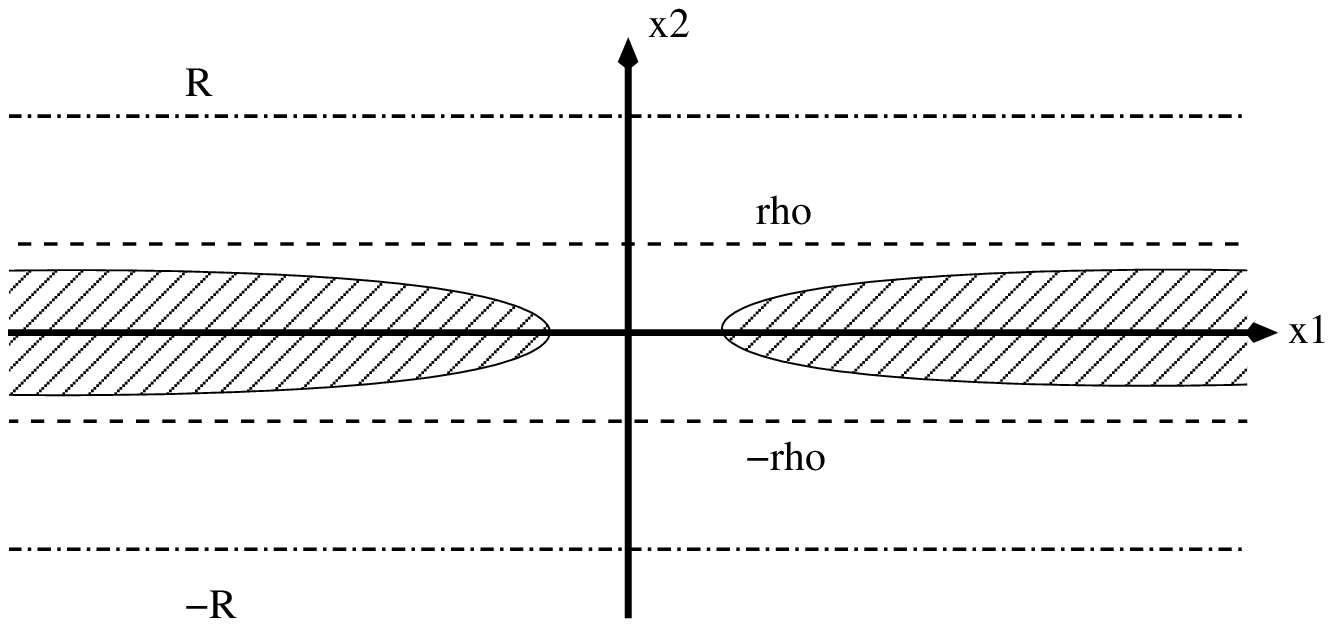}}
%
   \caption{(a) The $2\RR$-periodic function $\chi$ equals to one for $|x_2| \leq 2\rho $,
    and it vanishes at $\pm \RR$ up to order three. In this sketch, $\rho=1$
    and $R = 4$.
     (b) The support of the contrast (shaded) is included in $\Omega_\rho = \{ |x_2| < \rho \}$, 
     and $R>2\rho$.}
  \label{fig:2}
\end{figure}

Let us define a smoothed kernel $\K_{\smooth}$ by
\begin{equation}
  \label{eq:kSmooth}
  \K_{\smooth}(x) = \chi(x_2)\K_{\RR}(x)
  \qquad \text{for } x\in \R^2,
  \quad x \neq \left( \begin{matrix} 2\pi j_1 \\ 2\RR j_2 \end{matrix} \right), \, j \in\Z^2,
\end{equation}
where $\K_{\RR}$ is the kernel from~\eqref{eq:DefineKernel}.
Note that $\K_{\smooth}$ is $\alpha$-quasi-periodic in $x_1$,
$2\RR$-periodic in $x_2$, and a smooth function on its
domain of definition (that is, away from the singularity).

\begin{lemma}
  The integral operator $L_{\per}: \, L^2(\Omega_{\RR}, \C^2) \to H^1_{\per}(\Omega_{\RR})$
  defined by 
  \[
    L_{\per}f := \div \int_{\Omega_{\RR}} \K_{\smooth}(\cdot-y) f(y) \d{y} 
  \] 
  is bounded.
\end{lemma}
\begin{proof}
  We split the integral operator in two parts,
  \begin{align*}
    L_{\per}f & = \div \int _{\Omega_{\RR}} \K_{\smooth}(\cdot-y) f(y) \d{y}
    = \div \int _{\Omega_{\RR}} \chi(\cdot-y_2)\K_{\RR}(\cdot-y) f(y) \d{y}\\
    & = \div \int _{\Omega_{\RR}}\K_{\RR}(\cdot-y) f(y) \d{y}
    + \div \int _{\Omega_{\RR}} [\chi(\cdot-y_2)-1] \K_{\RR}(\cdot-y) f(y) \d{y}.
  \end{align*}
  By Theorem~\ref{th:boundKh}, the integral operator with
  the kernel $\K_{\RR}$ is bounded from
  $L^2(\Omega_{\RR}, \C^2)$ into $H^1_\alpha(\Omega_{\RR})$.
  Further, the definition of $\chi$ shows that
  $\chi(x_2 - y_2) = 1$ for $|x_2-y_2|\leq2\rho$.
  The kernel $(\chi-1)\K_{\RR}$ is hence smooth in
  $\Omega_{\RR}$, and the corresponding integral
  operator is compact from $L^2(\Omega_{\RR}, \C^2)$ into
  $H^1_\alpha(\Omega_{\RR})$. Hence, $L_{\per}$ is bounded
  from $L^2(\Omega_{\RR}, \C^2)$ into $H^1_\alpha(\Omega_{\RR})$.
  Periodicity of the kernel $\K_{\smooth}$ in the second component 
  of its argument finally implies that $L_{\per} f$ belongs to
  $H^1_{\per}(\Omega_{\RR}) \subset H^1_\alpha(\Omega_{\RR})$.
\end{proof}

Let us now consider the periodized integral equation
\begin{equation}
  \label{eq:PeriodicOperatorEquation}
  u - L_{\per}(q \nabla u) = L_{\per}(f) \qquad
  \text{in } H^1_{\per}(\Omega_{\RR}),
\end{equation}
where, for simplicity, we call the unknown function $u$.

\begin{theorem}
\label{th:equivalence}
(a) If $f \in L^2(D, \C^2)$, then $L_{\per}(f)$ equals $L(f)$ in $\Omega_\rho$.

(b) Equation~\eqref{eq:lippmannD} is uniquely solvable
in $H^1_\alpha(D)$ for any right-hand side $f \in L^2(D, \C^2)$
if and only if~\eqref{eq:PeriodicOperatorEquation} is uniquely
solvable in $H^1_{\per}(\Omega_{\RR})$ for any right-hand side
$f \in L^2(D, \C^2)$.

(c) If $q \in C^{2,1}(\ol{D})$ and if $f= q \nabla u^i$ for a 
smooth $\alpha$-quasi-periodic function $u^i$, then any solution 
to~\eqref{eq:PeriodicOperatorEquation} belongs to 
$H^{s}_{\per}(\Omega_R)$ for any $s<3/2$.
\end{theorem}

\begin{proof}
(a) For all $x$ and $y \in \Omega_{\RR}$
such that $|x_2-y_2|\leq 2\rho$ it holds that
$\K_{\smooth}(x-y) = \chi(x_2-y_2)\K_{\RR}(x-y) = G_{\alpha}(x-y)$.
In particular, for $x \in \Omega_\rho$ and $y \in D \subset \Omega_\rho$ 
it holds that $|x_2-y_2|\leq 2\rho$. Consequently, 
\begin{align*}
  (L_{\per}(f))(x)
   &= \div\int_{\Omega_{\RR}} \K_{\smooth}(x-y) f(y) \d{y}\\
  & = \div\int_{D} G_{\alpha}(x-y) f(y) \d{y}
  = (L(f))(x), \qquad x \in \Omega_\rho.
\end{align*} 

(b) Assume that $u^s \in H^1_\alpha(D)$ solves~\eqref{eq:lippmannD}
for a right-hand side $f \in L^2(D, \C^2)$ and define 
$\tilde{u} \in H^1_{\per}(\Omega_{\RR})$ by
$\tilde{u} = L_{\per}(q \nabla u^s + f)$. Since 
$u^s$ solves~\eqref{eq:lippmannD}, and due to 
part (a), we find that $\tilde{u}|_D = u^s$. Hence
$L_{\per}(q\nabla \tilde{u})=L_{\per}(q\nabla u^s)$
in $H^1_{\per}(\Omega_{\RR})$, which yields that
\begin{equation}
\label{eq:equiProof2}
 \tilde{u} = L_{\per}(q \nabla \tilde{u} +f) \text{ in } H^1_{\per}(\Omega_{\RR}).
\end{equation}
Now, if $f \in L^2(D, \C^2)$ vanishes, then uniqueness
of a solution to~\eqref{eq:lippmannD} implies
that $u^s \in H^1_\alpha(D)$ vanishes, too. Obviously,
$\tilde{u}= L_{\per}(q \nabla u^s)$ vanishes,
and hence~\eqref{eq:equiProof2} is uniquely solvable.  
The converse follows directly from (a).

(c) Assume that $u \in H^1_{\per}(\Omega_R)$ 
solves~\eqref{eq:PeriodicOperatorEquation} for $f= q \nabla u^i$.
Part (a) implies that the restriction of $u$ to $\Omega_\rho$ solves 
$u - L(q \nabla u) = L (q \nabla u^i)$ in $H^1_\alpha(\Omega_\rho)$.
Hence, Lemma~\ref{th:H2smoothness} implies that $u$ is a weak 
$\alpha$-quasi-periodic solution to $\div ((1+q) \nabla u) + k^2 u = - \div (q \nabla u^i)$
in $\Omega_\rho$. Transmission regularity results imply that $u$ belongs to 
$H^2_\alpha(D) \cap H^2_\alpha(\Omega_\rho \setminus \ol{D})$, and it is well-known 
that this implies that $u \in H^s_\alpha(\Omega_\rho)$ for $s<3/2$ 
(see, e.g.,~\cite[Section 1.2]{Grisv1992}). The function $u$ is even smooth 
in $\Omega_{\RR} \setminus \Omega_{\rho-\epsilon}$:
Recall that $\rho$ was chosen such that $\ol{D} \subset \Omega_{\rho}$.
Hence, there is $\epsilon>0$ such that $D \subset \Omega_{\rho-2\epsilon}$, and
\[
  u(x) = L_{\per}(q \nabla (u+u^i)) (x)
  = \div \int_{D} \K_{\smooth}(x-y) q(y) \nabla (u(y)+u^i(y)) \d{y}, 
  \quad x \in \Omega_{\RR} \setminus \Omega_{\rho-\epsilon}  
\]
shows that the restriction of $u$ to $\Omega_{\RR} \setminus \Omega_{\rho-\epsilon}$
is a smooth $\alpha$-quasi-periodic function, since the kernel of the above 
integral operator is smooth.
\end{proof}

Next we prove that the operator $I-L_{\per}(q \nabla \cdot)$
from~\eqref{eq:PeriodicOperatorEquation} satisfies a
G\r{a}rding inequality in $H^1_{\per}(\Omega_{\RR})$.
First, we announce a simple lemma that is useful in the next proof.

\begin{lemma}
\label{th:Compact1}
Suppose that $X$ and $Y$ are Hilbert spaces. Let $A$,
$B$ be bounded linear operators from $X$ into $Y$ and
consider the sesquilinear form $(u,v) \mapsto \langle Au,Bv\rangle_Y$ 
on $X \times X$.
If either $A$ or $B$ is compact, then
the linear operator $Q:\, X\rightarrow X$ defined by 
$\langle Qu,v\rangle_X = \langle Au,Bv\rangle_Y$ for 
$u,v \in X$ is compact, too.
\end{lemma}


\begin{theorem}
  \label{th:PeriodicGarding}
  Assume that $\sqrt{q} \in C^{2,1}(\ol{D})$,
  that $q \geq q_0 > 0$, 
  and that $D$ is of class $C^{2,1}$. Then there exists $C>0$ and a compact operator
  $K$ on $H^1_{\per}(\Omega_{\RR})$ such that
  \begin{equation}
    \label{eq:gardingPeriodic1}
    \Re\langle v - L_{\per}(q\nabla v),v\rangle_{H^1_{\per}(\Omega_{\RR})}
    \geq \|v\|^2_{H^1_{\per}(\Omega_{\RR})}
    - \Re\langle K v, \, v\rangle_{H^1_{\per}(\Omega_{\RR})}, 
    \quad v \in H^1_{\per}(\Omega_{\RR}).
  \end{equation}
\end{theorem}

\begin{remark}
  The idea of the proof is to split the integrals defining the
  inner product on the left of~\eqref{eq:gardingPeriodic1}
  into the three integrals on $D$, $\Omega_{\rho} \setminus \ol{D}$,
  and on $\Omega_{\RR} \setminus \ol{\Omega_\rho}$. For the
  term on $D$ one exploits the G\r{a}rding inequalities from
  Theorem~\ref{th:FreeGarding}.
  The terms on $\Omega_{\rho} \setminus \ol{D}$
  and on $\Omega_{\RR} \setminus \ol{\Omega_{\rho}}$ can be
  shown to be compact or positive perturbations.
\end{remark}

\begin{proof}
Let $v \in H^1_{\per}(\Omega_{\RR})$.
First, we split up the integrals arising from the inner product
on the left of~\eqref{eq:gardingPeriodic1} into
integrals on $D$, on $\Omega_{\rho} \setminus \ol{D}$, 
and on $\Omega_{\RR} \setminus \ol{\Omega_\rho}$.
Second, we use the G\r{a}rding inequality from
Theorem~\ref{th:FreeGarding} to find that 
\begin{multline}
  \label{eq:gardingPreparation}
  \Re\langle v -  L_{\per}(q \nabla v),v\rangle_{H^1_{\per}(\Omega_{\RR})}
  \geq \| v \|^2_{H^1_\alpha(D)} + \langle K v, \, v \rangle_{H^1_{\alpha}(D)}
  + \| v \|^2_{H^1_{\alpha}(\Omega_{\RR} \setminus \ol{D})} \\
  - \Re \big[ \langle L_{\per}(q \nabla v),v\rangle_{H^1_{\alpha}(\Omega_{\RR} \setminus \ol{\Omega_\rho})}
  + \langle L_{\per}(q \nabla v),v\rangle_{H^1_{\alpha}(\Omega_\rho \setminus \ol{D})} \big]
\end{multline}
with a compact operator $K$ on $H^1_\alpha(D)$.
Further, the evaluation of $L_{\per}(q\nabla \cdot)$ on
$\Omega_{\RR}\setminus \ol{\Omega_{\rho}}$ defines a
compact integral operator mapping $H^1_\alpha(D)$ to
$H^1_\alpha(\Omega_{\RR}\setminus \ol{\Omega_{\rho}})$,
because the (periodic) kernel of this integral operator is smooth.
(This argument requires the smooth kernel $\K_{\smooth}$ introduced
in the beginning of this section.) Lemma~\ref{th:Compact1} then 
allows to reformulate the corresponding term in~\eqref{eq:gardingPreparation}
in the way stated in the claim. Unfortunately, the last term
in~\eqref{eq:gardingPreparation} does not yield a compact 
sesquilinear form and needs a more detailed investigation.

For $x \in \Omega_{\rho}\setminus \ol{D}$ and $y \in D$ the
kernel $\K_{\smooth}(x-y)$ equals $G_{\alpha}(x-y)$, which
is a smooth function of $x \in \Omega_{\rho}\setminus \ol{D}$
and $y \in D$. Moreover, $\Delta G_{\alpha}(x-y) + k^2 G_{\alpha}(x-y)=0$
for $x \neq y$. Since $\nabla_x G_\alpha(x-y) = - \nabla_y G_\alpha(x-y)$, 
an integration by parts in $\Omega_{\rho}\setminus \ol{D}$ shows that 
\begin{align*}
  L(q\nabla v)(x) 
  = & \div \int_D G_{\alpha}(x-y)q(y)\nabla v(y)\d{y} \\
  = & - \int_D \nabla_y G_{\alpha}(x-y)\cdot\nabla (qv)(y)\d{y}
  + \int_D \nabla_y G_{\alpha}(x-y)\cdot \nabla q(y) v(y)\d{y} \\
  = & - k^2\int_D G_{\alpha}(x-y) q(y)v(y)\d{y} - L(v\nabla q)(x)  \\
  & \qquad - \int_{\partial D} \frac{\partial G_{\alpha}(x-y)}{\partial \nu(y)} \gamma_{\mathrm{int}}(q)(y) \gamma(v)(y)\d{s}
  \qquad \text{for } x \in \Omega_{\rho}\setminus \ol{D}, 
\end{align*}
where $\nu$ is the exterior normal vector to $D$.
The integral operator appearing in the last term of the
last equation is the double layer potential $\mathrm{DL}$,
\[
  \mathrm{DL}(\psi)
  = \int_{\partial D} \frac{\partial G_{\alpha}(\cdot-y)}{\partial \nu(y)} \psi(y) \d{s} \quad
  \text{in } \Omega \setminus \partial D.
\]
It is well-known that $\mathrm{DL}$ defines a bounded operator from
$H^{1/2}_\alpha(\partial D)$ into $H^1_\alpha(\Omega_{\RR}\setminus \ol{D})$ and
into $H^1_\alpha(D)$ (see, e.g.,~\cite{Arens2006}).
This implies that the jump of the double-layer potential
\[
  T \psi := [\mathrm{DL} \psi]_{\partial D}
  = \gamma_{\mathrm{ext}}(\mathrm{DL} \psi) - \gamma_{\mathrm{int}}(\mathrm{DL} \psi)
\]
from the outside of $D$ to the inside of $D$ is a bounded operator
on $H^{1/2}_\alpha(\partial D)$. It is well-known that in our case
$T$ is even a compact operator on $H^{1/2}_\alpha(\partial D)$, since $D$ is
of class $C^{2,1}$. Additionally, the equality
$\gamma_{\mathrm{int}} (\mathrm{DL} \psi) = - \psi / 2 + T \psi$
holds for $\psi \in H^{1/2}_\alpha(\partial D)$.

For $v \in H^1_{\per}(\Omega_{\RR})$,
\begin{equation}
  \label{eq:gardingPreparation2}
  -\langle \nabla L(q\nabla v), \, \nabla v \rangle_{L^2(\Omega_{\rho}\setminus \ol{D})}
  = \langle k^2\nabla V(q v) + \nabla L(v\nabla q) + \nabla \mathrm{DL}(\gamma_{\mathrm{int}}(q v)), \, \nabla v\rangle_{L^2(\Omega_{\rho}\setminus \ol{D})}.
\end{equation}
The mapping properties of $V$ shown in
Lemma~\ref{th:H2smoothness} and the smoothness of $q$
imply that $v \mapsto k^2\nabla V(q v) + \nabla L(v\nabla q)$
is compact from $H^1_{\per}(\Omega_{\RR})$ into $L^2(D)$.
To finish the proof we show that the last term
in~\eqref{eq:gardingPreparation2} can be written as a sum
of a positive and compact term. For simplicity, we define
$w = \mathrm{DL}(\gamma_{\mathrm{int}}(q v))$ and note that 
$-  v/2 = [\gamma_{\mathrm{int}}(w) - T(\gamma_{\mathrm{int}}(q v))]/\gamma_{\mathrm{int}}(q)$
on $\partial D$. Since it plays no role whether the normal derivative 
$\partial w / \partial \nu$ is taken from the inside or from the outside of $D$, 
we skip writing down the trace operators for the normal derivative. Then
\begin{multline}
  \label{eq:DoubleLayer}
  \langle \nabla \mathrm{DL}(q v),\nabla v\rangle_{L^2(\Omega_{\rho}\setminus \ol{D})}
  = \int_{\Omega_{\rho}\setminus \ol{D}} \nabla w\cdot\nabla\overline{v} \d{x}  \\
  = k^2\int_{\Omega_{\rho}\setminus \ol{D}} w \overline{v} \d{x}  
    - \int_{\partial D}\frac{\partial w}{\partial \nu} \overline{v} \d{s}
    + \int_{\Gamma_{\rho}} \frac{\partial w}{\partial x_2}\overline{v} \d{s}
    - \int_{\Gamma_{-\rho}} \frac{\partial w}{\partial x_2}\overline{v} \d{s}
\end{multline}
and the above jump relation shows that
\begin{align*}
  - \frac{1}{2} \int_{\partial D}\frac{\partial w}{\partial \nu}\overline{v} \d{s}
  & = \int_{\partial D}\frac{\partial w}{\partial \nu}\frac{\gamma_{\mathrm{int}}(\overline{w})}{\gamma_{\mathrm{int}}(q)} \d{s}
    - \int_{\partial D}\frac{\partial w}{\partial \nu}\frac{\overline{T(\gamma_{\mathrm{int}} (qv))}}{\gamma_{\mathrm{int}}(q)} \d{s} \\
  & = \int_{D} \nabla w \cdot \nabla \bigg( \overline{\frac{w}{q}} \bigg) \d{x}
    + \int_{D} \Delta w \overline{\frac{w}{q}}\d{x}
    - \int_{\partial D}\frac{\partial w}{\partial \nu} \frac{\overline{T(\gamma_{\mathrm{int}} (qv))}}{\gamma_{\mathrm{int}}(q)}  \d{s} \\
  & = \int_{D}\frac{|\nabla w|^2}{q}\d{x}
    +  \int_{D}\big(\nabla q^{-1} \cdot \nabla w - k^2\frac{w}{q}\big) \overline{w}\d{x} 
    - \int_{\partial D}\frac{\partial w}{\partial \nu} \frac{\overline{T(\gamma_{\mathrm{int}} (qv))}}{\gamma_{\mathrm{int}}(q)}  \d{s}.
\end{align*}
Combining the last computation with~\eqref{eq:DoubleLayer} shows that
\begin{align}
  \label{eq:positive}
  & \big\langle \nabla \mathrm{DL}(q v|_{\partial D}), \,
    \nabla v|_{\Omega_{\rho}\setminus \ol{D}} \big\rangle_{L^2(\Omega_{\rho}\setminus \ol{D})}
  =  2\int_{D} \frac{|\nabla w|^2}{q}\d{x} + k^2\int_{\Omega_{\rho}\setminus \ol{D}} w\overline{v}\d{x} \\
  & + 2 \int_{ D} \left( \nabla q^{-1} \cdot \nabla w - k^2\frac{w}{q} \right) \overline{w} \d{x} 
    - 2\int_{\partial D}\frac{\partial w}{\partial \nu} \frac{\overline{T(\gamma_{\mathrm{int}} (qv))}}{\gamma_{\mathrm{int}}(q)}  \d{s}
    + \left( \int_{\Gamma_{\rho}}-\int_{\Gamma_{-\rho}} \right) \frac{\partial w}{\partial x_2}\overline{v} \d{s}. \nonumber
\end{align}
Using Lemma~\ref{th:Compact1}, all the terms in the second line of 
the last equation can be rewritten as 
$\langle K_1 v, \, v \rangle_{H^1_\per(\Omega_{\RR})}$ where 
$K_1$ is a  compact operator on $H^1_\per(\Omega_{\RR})$. 
The mapping $v \mapsto \int_{D} |\nabla w|^2 / q \d{x}$ is obviously
positive if $q > 0$. In consequence,~\eqref{eq:gardingPreparation}
and~\eqref{eq:gardingPreparation2} show that~\eqref{eq:gardingPeriodic1} holds.
 \end{proof}

\section{Discretization of the Periodic Integral Equation}
\label{se:discretization}

In this section we firstly consider the discretization of the periodized
integral equation~\eqref{eq:PeriodicOperatorEquation} in spaces
of trigonometric polynomials. If the periodization satisfies certain 
smoothness conditions and if uniqueness of solution holds, 
convergence theory for the discretization is a consequence of 
the G\r{a}rding inequalities shown in Theorem~\ref{th:PeriodicGarding}.
Secondly we present fully discrete formulas for implementing a Galerkin 
discretization of the Lippmann-Schwinger integral 
equation~\eqref{eq:PeriodicOperatorEquation}.

For $N \in \N$ we define
$\mathbb{Z}^2_N = \{j \in \mathbb{Z}^2: \, - N/2 <j_{1,2} \leq N/2 \}$
and $\T_N = \mathrm{span} \{ \varphi_j: \, j\in \mathbb{Z}^2_N \}$,
where $\varphi_j \in L^2(\Omega_{\RR})$ are the $\alpha$-quasi-periodic
basis functions from~\eqref{eq:trigonometricPolynomial}. 
Note that the union $\cup_{N\in\N} \T_N$ is dense in $H^1_{\per}(\Omega_{\RR})$.
The orthogonal projection onto $\T_N$ is 
\[
  P_N: \, H^1_{\per}(\Omega_{\RR}) \to \T_N, \qquad 
  P_N(v) = \sum_{j\in \mathbb{Z}_N^2} \hat{v}(j) \varphi_j,
\]
where $\hat{v}(j)$ denotes as above the $j$th Fourier coefficient.
The next proposition recalls the standard convergence 
result for Galerkin discretizations of equations that satisfy 
a G\r{a}rding inequality, see, e.g.~\cite[Theorem 4.2.9]{Saute2007}, combined 
with the regularity result from Theorem~\ref{th:equivalence}(c).

\begin{proposition}
  \label{th:convergence}
  Assume that $q$ satisfies the assumptions of
  Theorem~\ref{th:PeriodicGarding} and 
  that~\eqref{eq:lippmannD} is uniquely solvable. 
  Then~\eqref{eq:PeriodicOperatorEquation} has a 
  unique solution $u \in H^1_{\per}(\Omega_{\RR})$, and
  then there is $N_0 \in \N$ such that the finite-dimensional
  problem to find $u_N \in \T_N$ such that
  \begin{align}
    \langle u_N - L_{\per}(q\nabla u_N), w_N \rangle_{H^1_{\per}(\Omega_{\RR})}
    = \langle f, w_N \rangle_{H^1_{\per}(\Omega_{\RR})}
    \quad \text{for all $w_N \in \T_N$} 
  \end{align}
  possesses a unique solution for all $N\geq N_0$ 
  and $f \in H^1_{\per}(\Omega_{\RR})$. In this case
  \[
    \| u_N  - u \|_{H^1_{\per}(\Omega_{\RR})}
    \leq C \inf_{w_N \in \T_N} \| w_N -u \|_{H^1_{\per}(\Omega_{\RR})}
    \leq C N^{-s} \| u \|_{H^{1+s}_{\per}(\Omega_{\RR})}, \quad 0 \leq s < 1/2,
  \]
  with a constant $C$ independent of $N \geq N_0$.
\end{proposition}
\begin{remark}
\label{th:RemarkConvergence}
  The convergence rate increases to $s+1-t$ if one measures the error in the 
  weaker Sobolev norms of $H^t_{\per}(\Omega_{\RR})$, $1/2<t<1$.
  This can be shown using adjoint estimates (see, e.g.~\cite[Section 4.2]{Saute2007} 
  for the general technique). However, the (linear) rate saturates at $t=1/2$, 
  since the integral operator is not bounded on $H^t_{\per}(\Omega_{\RR})$ 
  for $t < 1/2$, that is, the $L^2$-error decays with a linear rate.
\end{remark}

Applying $P_N$ to the infinite-dimensional
problem~\eqref{eq:PeriodicOperatorEquation}, and 
exploiting that $P_N$ commutes with the periodic 
convolution operator $L_{\per}$, we obtain 
the discrete problem to find $u_N \in \T_N$ such that
\begin{equation}
  \label{eq:DiscreteEquation}
  u_N - L_{\per}(P_N(q\nabla u_N)) = L_{\per}(P_N f).
\end{equation}
Fast methods to evaluate the discretized operator
in~\eqref{eq:DiscreteEquation} exploit that the
application of $L_{\per}$ to a trigonometric polynomial
in $\T_N$ can be explicitly computed using an
$\alpha$-quasi-periodic discrete Fourier transform that 
we call $\F_{N}$. This transform maps point values of
a trigonometric polynomial $\varphi_j$ 
(see~\eqref{eq:trigonometricPolynomial}) to the 
Fourier coefficients of the polynomial.
If $h := (2\pi/N, \, 4\pi\RR/N)^\top$ (a column vector), then
\[
  \hat{v}_N(j) = \frac{\sqrt{4\pi\RR}}{N^2}
  \sum_{l\in \mathbb{Z}_N^2} v_N(l \cdot h)
  \exp\big(- 2\pi\i \,  (j_1+\alpha,j_2)^\top \cdot l /N\big),
  \qquad j \in \Z^2_N.
\]
(Vectors $j,l \in \mathbb{Z}_N^2$ are interpreted as a column vectors.)
This defines the transform $\F_{N}$ mapping
$(v_N(j \cdot h))_{j\in\Z^2_N}$ to $(\hat{v}_N(j))_{j\in \Z^2_N}$.
The inverse $\F_{N}^{-1}$ is explicitly given by
\[
  v_N(j \cdot h) = \frac{1}{\sqrt{4\pi\RR}} \sum_{l\in \mathbb{Z}_N^2}
  \hat{v}_N(l) \exp\left(2\pi\i \, (l_1+\alpha,l_2)^\top \cdot j/N \right), \qquad j \in \Z^2_N.
\]
Both $\F_{N}$ and its inverse are linear operators on 
$\C^2_N = \{ (c_n)_{n \in \Z^2_N}: \, c_n \in \C \}$.
The restriction operator $R_{N,M}$ from $\C^2_{N}$ to 
$\C^2_{M}$, $N > M$, is defined by
$R_{N,M}(a) = b$ where $b(j) = a(j)$ for $j\in\Z^2_M$.
The related extension operator $E_{M,N}$ from $\C^2_{M}$ to 
$\C^2_{N}$, $M<N$, is defined by $E_{M,N}(a) = b$ where 
$b(j) = a(j)$ for $j\in\Z^2_M$ and $b(j) = 0$ else.

For the next lemma, we introduce the notation $A \bullet B = (A_{ij} B_{ij})_{i,j =1}^M$ 
for the componentwise product of two matrices $A,B \in \C^{M\times M}$.

\begin{lemma}
  \label{eq:coeffMultipli}
  The Fourier coefficients of $q \partial_{\ell} u_N$, $\ell=1,2$, are given by
  \[
    (\widehat{q\partial_{\ell} u}_N(j))_{j\in \Z^2_N}
    = R_{3N,N} \mathcal{F}_{3N}\big[
      \mathcal{F}_{3N}^{-1}\big(E_{2N,3N} (\hat{q}_{2N}(j))_{j\in\Z^2_N} \big)
      \bullet
      \mathcal{F}_{3N}^{-1}\big(E_{N,3N} (w_{\ell}(j) \hat{u}_N(j))_{j\in\Z^2_N} \big) \hspace*{-0.5mm} \big]
  \]
  where $w_1(j) = \i (j_1+\alpha)$ and $w_2(j) = \i j_2 \pi /\RR$ for $j\in \Z^2$.
\end{lemma}
\begin{proof}
  For $u_N \in \T_N$, $j \in \Z^2$, and $\ell=1,2$,
  \begin{align}
    \label{eq:convolution}
      4\pi\RR \,\widehat{q\partial_\ell u}_N (j) 
       &  = 4\pi\RR \int_{\Omega_{\RR}}q\partial_\ell u_N\ol{\varphi_j}\d{x}
      = 4\pi\RR \sum_{m\in\Z^2_N}\widehat{\partial_\ell u}_N(m)\int_{\Omega_{\RR}}q\ol{\varphi_j}\varphi_m\d{x}\\
      & =  \sum_{m\in\Z^2_N}\widehat{\partial_\ell u}_N(m) 
        \int_{\Omega_{\RR}} q(x) e^{-\i[(j_1 -m_1)x_1 + (j_2-m_2) x_2 \pi/\RR]} \d{x} \nonumber \\
      & = (4\pi\RR)^{1/2} \sum_{m\in\Z^2_N} \widehat{\partial_\ell u}_N(m) \hat{q}(j-m). \nonumber
  \end{align}
  If $j \in\Z^2_N$, then the coefficient $\widehat{q\partial_\ell u_N}(j)$
  merely depends on $\hat{q}(m)$ for $m \in \Z^2_{2N}$. Hence,
  $\widehat{q\partial_\ell u_N }(j) = \widehat{q_{2N} \partial_\ell u_N} (j)$
  for $j\in \Z^2_N$. Obviously, $q_{2N} \partial_\ell u_N$ belongs to
  $\T_{3N}$. Hence, the Fourier
  coefficients of $q_{2N}\partial_\ell u_N$ are given by
  $\mathcal{F}_{3N}$ applied to the grid values of this
  function at $j \cdot h$, $j\in \mathbb{Z}_{3N}^2$.
  The grid values of $\widehat{\partial_\ell u_N}$ are given by
  $\mathcal{F}_{3N}^{-1} (E_{N,3N} (\widehat{\partial_\ell u}_N(j)_{j\in\Z^2_N} )$,
  and the grid values of $q_{2N}$ can be computed analogously.
  Finally, taking a partial derivative with respect to $x_1$
  or $x_2$ of $u$ yields a multiplication of the $j$th Fourier
  coefficient $\hat{u}(j)$ by $\i (j_1+\alpha)$ and
  $\i j_2 \pi /\RR$, respectively.
\end{proof}

In Lemma~\ref{th:boundKh} we computed the Fourier coefficients of
the kernel $\K_{\RR}$. The kernel $\K_{\smooth}$ used to define
the periodized potential $L_{\per}$ is the product of $\K_{\RR}$
with the smooth function $\chi$ (see~\eqref{eq:kSmooth}). Hence,
the Fourier coefficients of $\K_{\smooth}$ are convolutions of
the $\hat\K_{\RR}(j)$ with
$\hat{\chi}(j_2) = (4\pi\RR)^{-1/2} \int_{-\RR}^\RR \exp(-\i j_2 \pi x_2/R) \chi(x_2) \d{x_2}$,
\[
  \hat{\K}_\smooth(j) = \frac{1}{(4\pi\RR)^{1/2}}
  \sum_{m\in\Z^2_N} \hat{\K}_{\RR}(j_1,m_2)\hat{\chi}(j_2-m_2), \qquad j \in \Z^2.
\]
The latter formula can be seen by a computation similar
to~\eqref{eq:convolution}. Note that $\chi$ is a smooth function, 
which means that the Fourier coefficients $\hat{\chi}$ in the 
last formula are rapidly decreasing, that is, the truncation
the last series converges rapidly to the exact value.
The convolution structure of $L_{\per}$ finally shows that
\begin{equation}
  \label{eq:evaluateL}
  \widehat{(L_{\per} f)}(j)
  = (4\pi\RR)^{1/2} \, \hat{\K}_\smooth(j)
  \Big[\i(j_1 + \alpha) \hat{f_1}(j) + \frac{\i j_2 \pi}{\RR}\hat{f_2}(j) \Big],
  \quad 
  f = \left( \begin{matrix} f_1 \\ f_2 \end{matrix} \right) \in L^2(\Omega_{\RR}, \C^2).
\end{equation}

The finite-dimensional operator $u_N \mapsto L_{\per}(P_N(q\nabla u_N))$
can now be evaluated in $O(N \log(N))$ operations by combining the 
formula of Lemma~\ref{eq:coeffMultipli} with~\eqref{eq:evaluateL}. 
The linear system~\eqref{eq:DiscreteEquation} can 
then be solved using iterative methods. Whenever one 
uses iterative techniques, one would of course like to precondition the 
linear system. The usual multi-grid preconditioning technique for integral 
equations of the second kind (see, e.g.,~\cite{Vaini2000}) does not 
apply here, since the integral operator is not compact. For the numerical 
experiments presented in the following section, we simply used the 
(unpreconditioned) GMRES algorithm from~\cite{Kelle1995}.

\section{Numerical Examples}
\label{se:numerics}

In this section we present several numerical examples that first 
check the quasi-optimal convergence rate of the trigonometric 
Galerkin scheme under investigation from 
Proposition~\ref{th:convergence}. The second aim of the examples 
is to show that the scheme is able to cope with spatially varying 
material configurations and hence can be applied when, e.g., boundary 
integral techniques are blocked. Third, we try to illustrate memory needs 
and computation times to give an impression about the performance of the 
scheme. Before going into details, let us emphasize again that one of the main 
advantages of the trigonometric Galerkin scheme is its straightforward and 
rather easy implementation, at least if a performant FFT routine and an 
efficient linear solver are already at hand.  
 
We first present a computational example for a very simple strip-like structure for 
which one can compute the scattered field for an incident plane wave analytically.
Checking that the numerically computed solution to the scattering problem converges 
well to the analytically known expression will be the first test for the correctness of the 
code (and a -- due to its simplicity limited -- test for the convergence rate, too).
Recall that we aim to compute the scattered field for an 
incident field $u^i(x_1,x_2) = \exp(\i k(\cos(\theta)x_1 - \sin(\theta)x_2))$
with incident angle $\theta$. For this first test we choose $k = \pi/2$ and  $\theta = \pi/4$
and approximate the solution in $\mathcal{T}_N$ where $N = 2^n$ for $n = 6,...,11$. 
For this example, $D = (-\pi,\pi)\times(-0.75,0.75)$ is a strip, we choose 
$\Omega_R = (-\pi,\pi)\times(-2,2)$, and the contrast $q$ equals two in $D$ 
(compare Figure~\ref{fig:stripExperiment}(a)). For this setting one can 
explicitly compute the scattered field and use the analytic expression for comparison.  
In the Figure~\ref{fig:stripExperiment}(b) we show the relative error between the numerical and the 
analytical solution in the norms $H^s_{\per}(\Omega_R)$ where $s = 0, 0.5, 1$. The relative error 
measured in the norm $H^1_{\per}(\Omega_R)$ fits quite well to the theoretical statement 
in Proposition~\ref{th:convergence}. Furthermore, if one measures the relative error 
in the norm $H^s_{\per}(\Omega_R)$ for $s = 0$ and $s = 0.5$ the experiment confirms 
the statement of Remark~\ref{th:RemarkConvergence}. 
All computations in this and in all following numerical experiments were done on a machine
with an Intel Xeon 3200 quad core processor and 12 GB memory. The trigonometric 
Galerkin discretization of the volume integral equation was implemented in MATLAB, 
relying on FFTW routines~\cite{Frigo2005} that MATLAB is able to execute in parallel. 
The linear system was solved by the GMRES iteration 
from~\cite{Kelle1995}. The iteration was stopped when the relative residual reduction 
factor was less than $10^{-5}$ (this parameter was chosen for all later experiments). 
Figure~\ref{tab:timesStrip} shows computation times and 
the number of GMRES iterations for this numerical test. Obviously, the computation 
time of the scheme gets large when $N$ becomes very large due to memory needs. 
The number of GMRES iterations slowly decreases in $N$ from $7$ to $5$.

\begin{figure}[h!!!!t!!b]
    \centering
    \begin{tabular}{c c}
      \hspace*{-5mm}
      \includegraphics[height = 6.5cm]{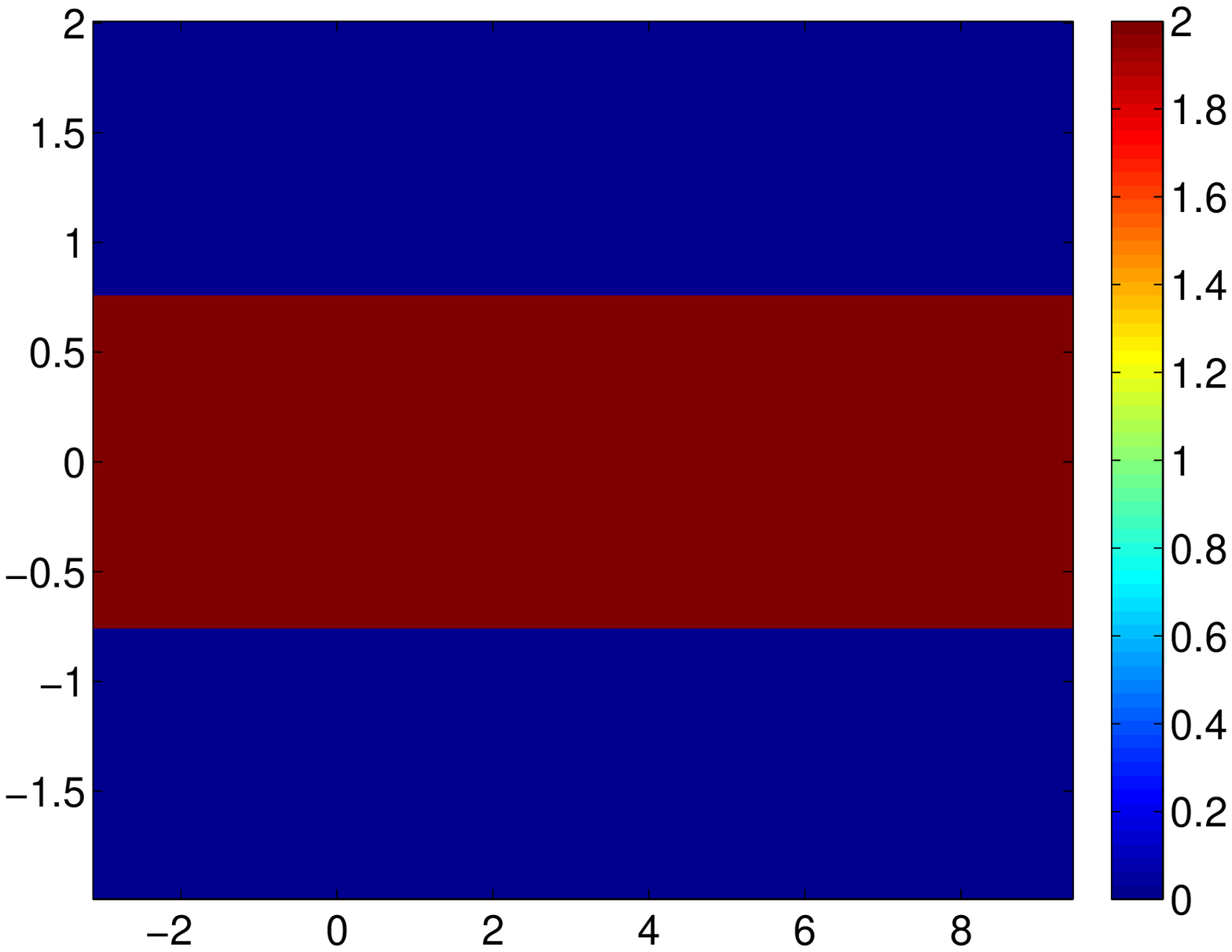}  & \includegraphics[height =6.5cm]{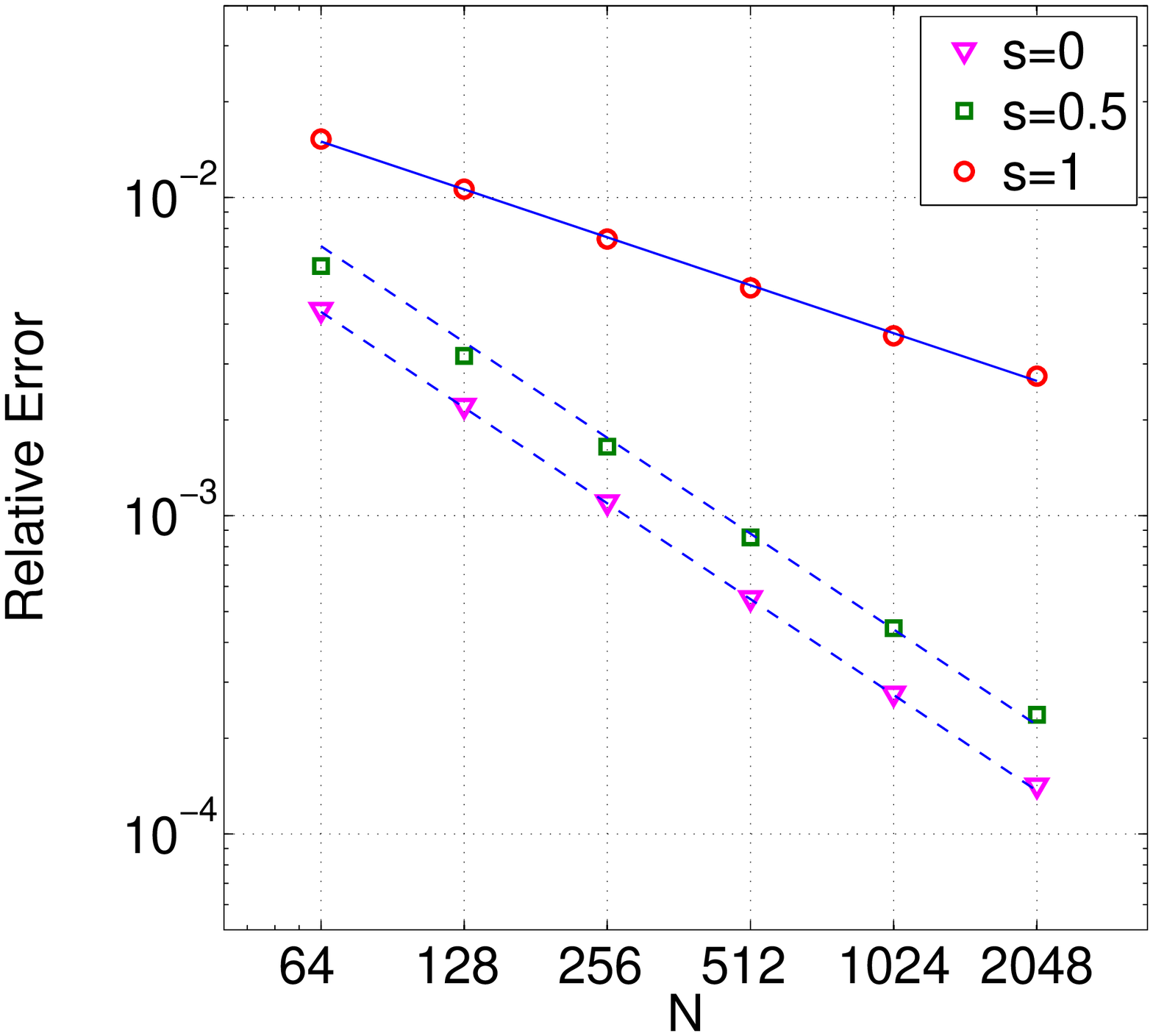}\\ 
       (a) & (b)
    \end{tabular}
    \caption{(a) Two periods (in the horizontal variable) of the strip structure with contrast equal to two. 
    (b) Relative error between the numerically approximated solution and the analytically computed reference solution measured 
in $H^s_{\per}$-norm for scattering from the strip shown in (a). Circles, kites, triangles correspond to $s = 1$, $s = 0.5$ 
and $s = 0$, respectively. The continuous line and the dotted lines indicate the convergence order 
0.5 and 1, respectively. The discretization parameter is $N = 2^n$ for $n = 6,...,11$.}
     \label{fig:stripExperiment}
\end{figure}

\begin{table}[h!!!!t!!!b!]
 \centerline{
 \begin{tabular}{ r | c c c c c c }\hline
 $N$                & 64 & 128 & 256 & 512 & 1024 & 2048 \\ \hline 
 Computation time(s) for strip-structure & 0.3 & 1.1 & 3.7 & 21.2 & 131.7 & 463.7 \\ 
 $\sharp$ GMRES iterations for strip-structure & 7 & 6 & 6 & 6 & 6 & 5 \\ \hline
 \end{tabular}}
 \caption{Computation times and number of GMRES iterations for the computation 
 of the errors for the simple strip structure shown in Figure~\ref{fig:stripExperiment}. 
 The parameter $N$ is the discretization parameter of the trigonometric Galerkin scheme.} 
 \label{tab:timesStrip}
\end{table}

Of course, the numerical results for the strip structure from the last example merely 
provide a first test that the algorithm computes correct solutions. For further tests and 
illustrations of the algorithm, we consider more complicated structures where the contrast
varies smoothly within subdomains and jumps across subdomain borders. As we mentioned 
in the introduction and confirmed in Lemma~\ref{eq:coeffMultipli}, it is essential for the 
Galerkin scheme to have explicit values of the Fourier coefficients $\hat{q}_{2N}$ of the 
contrast $q$ at hand. In principle, these values could be approximated using FFTs. 
However, we found that whenever one is able to compute these Fourier coefficients 
analytically, this results in considerably more accurate computations. In the examples 
below, we explain case-by-case how to compute these Fourier coefficients for a wide class of polynomially 
or exponentially varying materials. For complicated material shapes, it is usually impossible 
to compute the Fourier coefficients explicitly. Using partial integrations, one is however 
able to come up with semi-analytic expressions that merely require a one-dimensional 
integration of a periodic and piecewise analytic function for evaluation.   

\begin{figure}[!ht]
    \centering
    \begin{tabular}{c c}
      \hspace*{-10mm}  \includegraphics[height = 6.5cm]{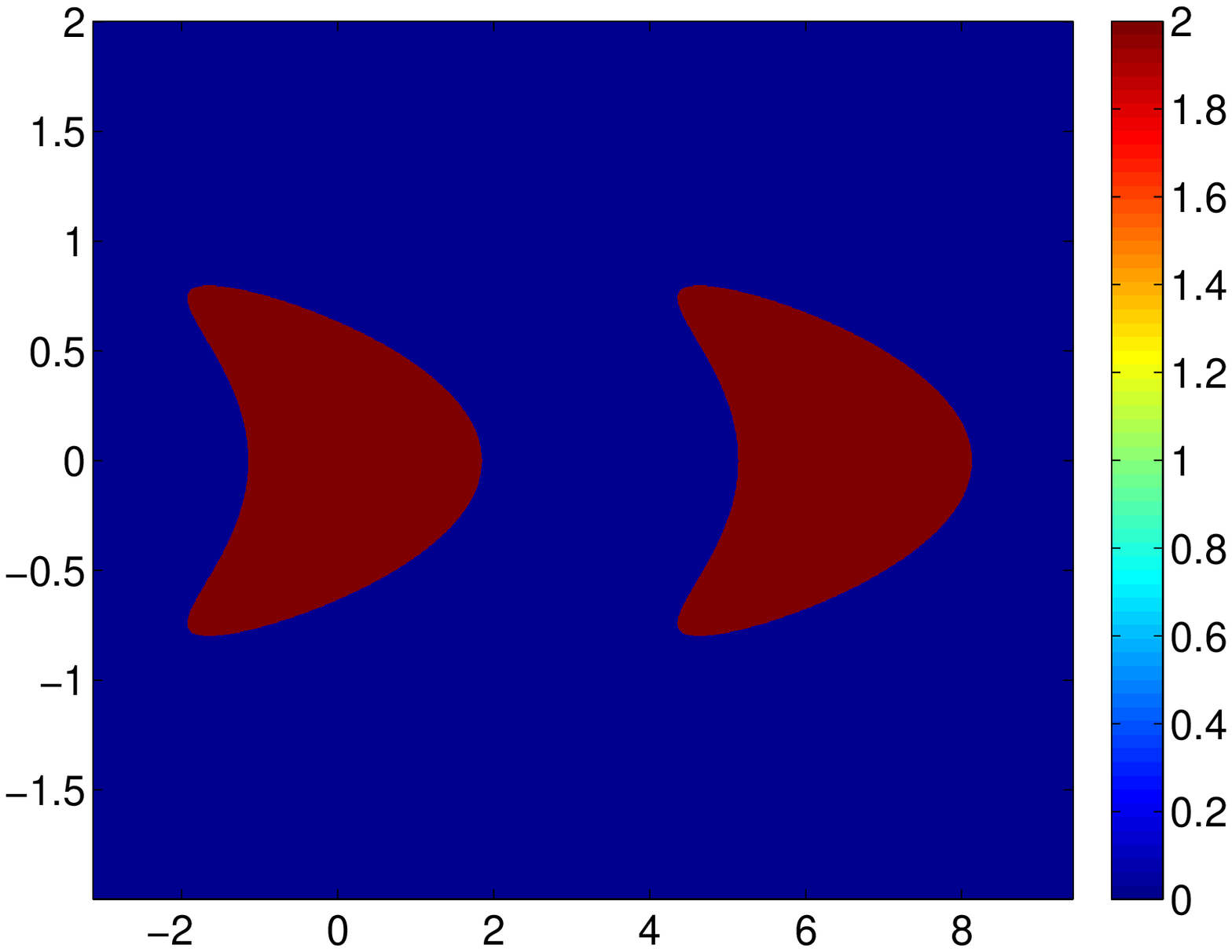}  \hspace*{-5mm} & \includegraphics[height = 6.5cm]{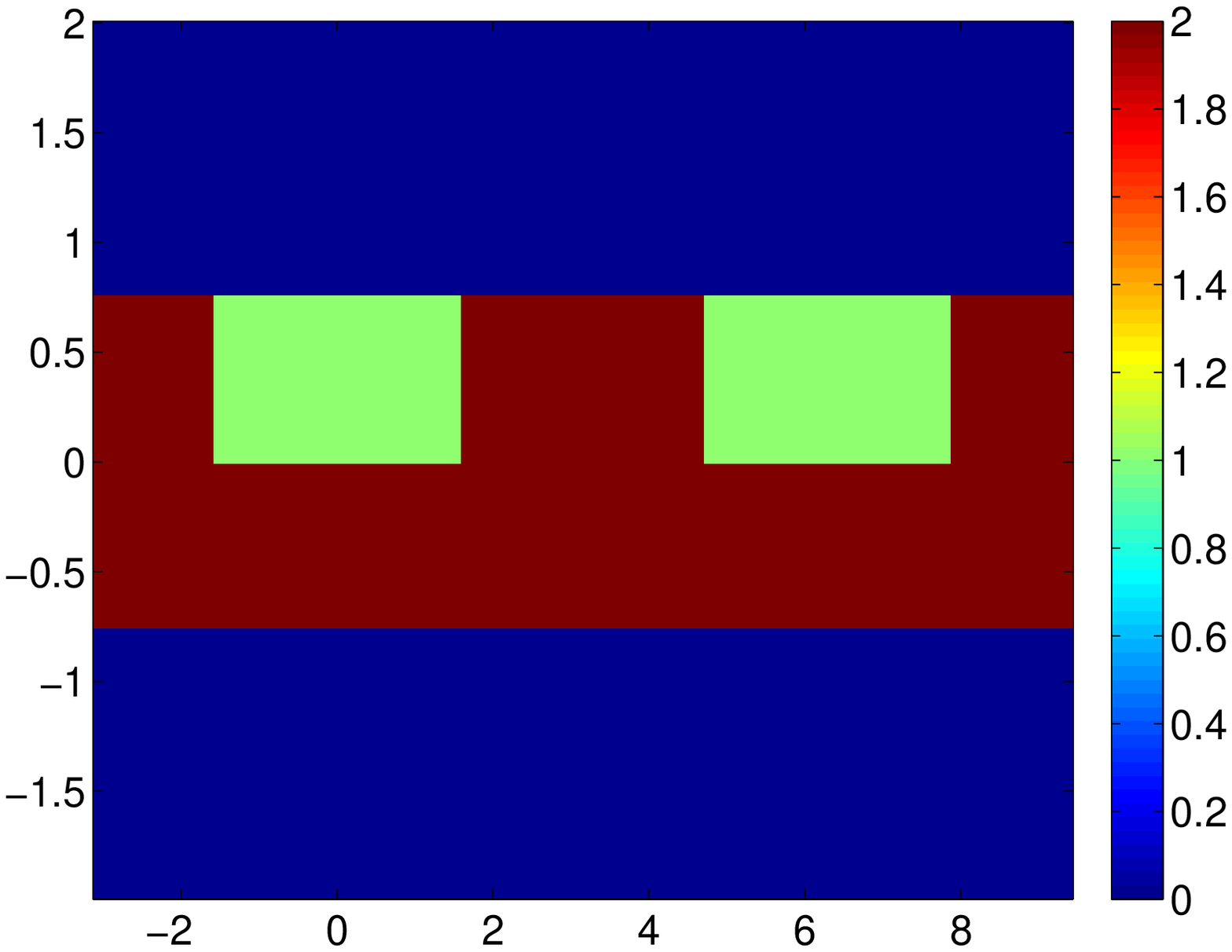} \\ 
       (a) & (b) \\
       \hspace*{-10mm} \includegraphics[height = 6.5cm]{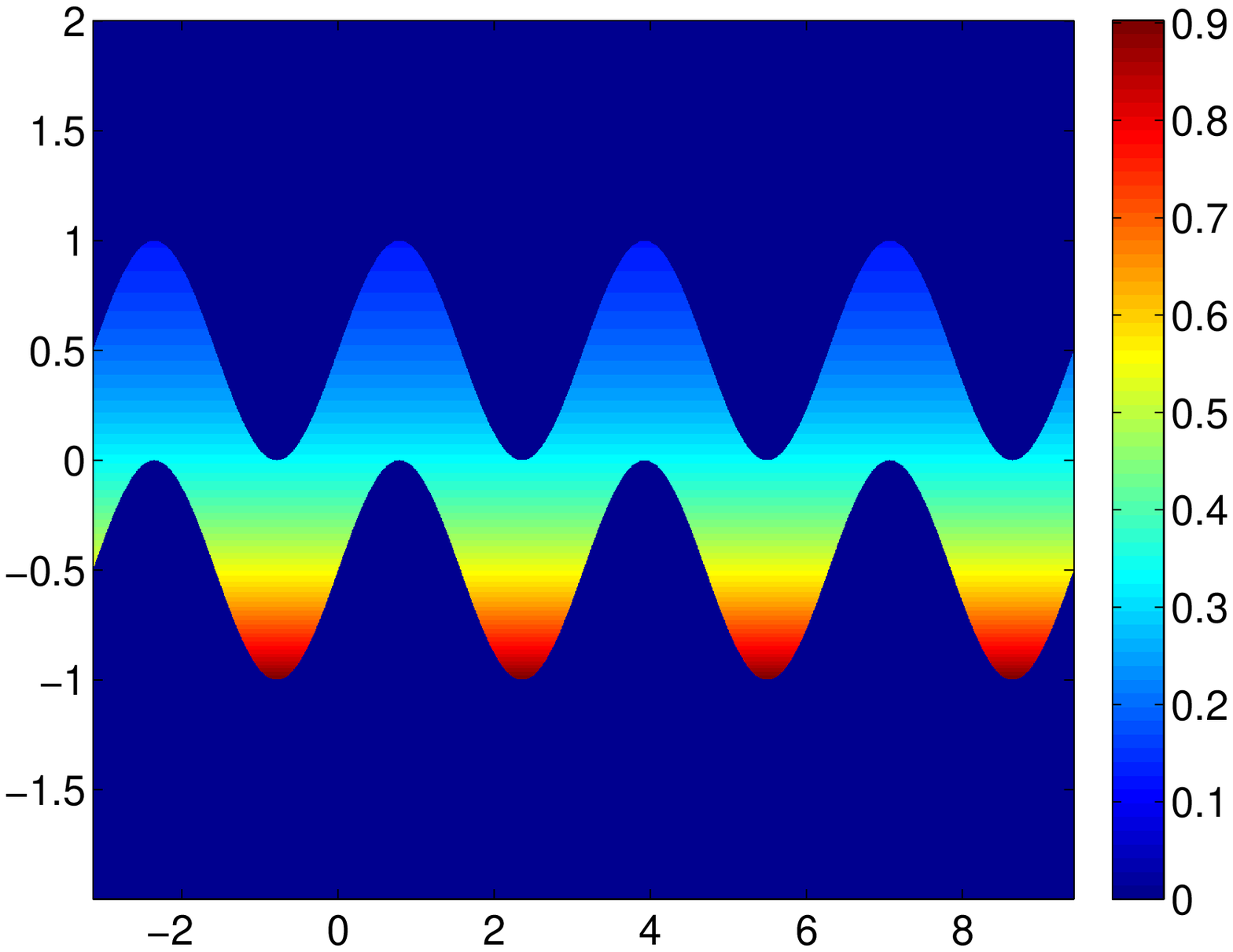}  \hspace*{-5mm} & \includegraphics[height = 6.5cm]{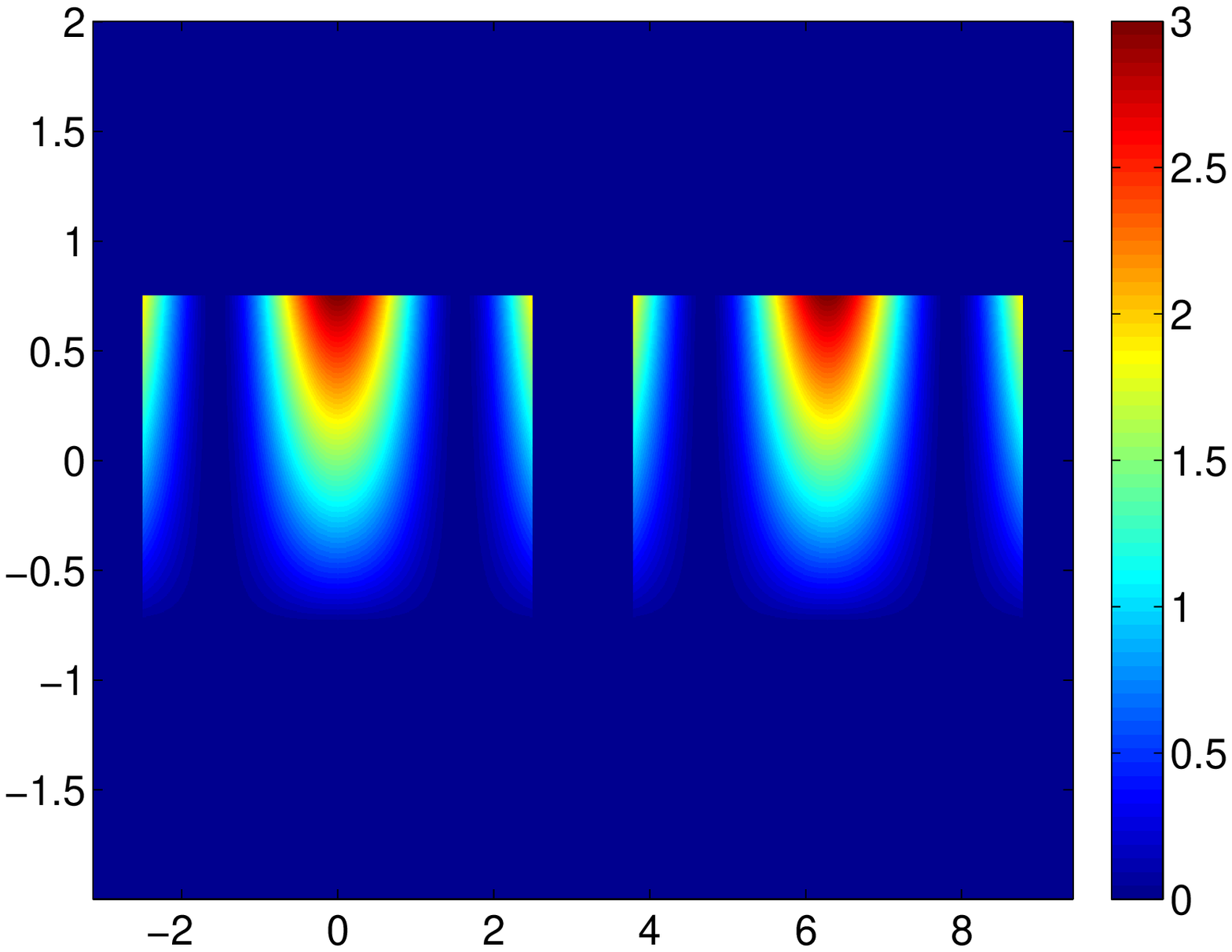} \\ 
        (c) & (d)
    \end{tabular}
        \caption{The four plots show two periods in the horizontal variable 
        of the contrasts $q_{1,2,3,4}$ considered in the numerical experiments below. 
        (a) The piecewise constant kite-shaped contrast $q_1$. 
        (b) The piecewise constant contrast $q_2$ is supported in a strip. 
        (c) The contrast $q_3$ varies smoothly within a sinusoidally-shaped strip.
        (d) The contrast $q_4$ varies smoothly within a rectangle.}
        \label{fig:contrasts}
\end{figure}

Figure~\ref{fig:contrasts} shows the four contrasts $q_{1,2,3,4}$ that 
we consider in the experiments below. We start now by giving precise definitions 
of these four contrasts and we compute their Fourier coefficients (semi-)explicitly.
Afterwards, we present numerical examples for the different structures.
We would like to point out in advance that for all four examples the domain 
$\Omega_\RR$ will be chosen as $(-\pi,\pi)\times(-2,2)$, i.e., $\RR = 2$ always.

The contrast $q_1$ plotted in Figure~\ref{fig:contrasts}(a) consists of 
$2\pi$-periodically aligned kite-shaped inclusions with constant 
material parameter (the contrast equals to two inside the inclusion). 
The boundary of the central inclusion $D \subset (-\pi,\pi) \times (-2,2)$ is 
parametrized  by $t \mapsto (1.5\cos(t) + \cos(2t) -0.65, \sin(t))^\top, \ t \in [0,2\pi]$. 
The Fourier coefficients $\{ \hat{q}_1(j) \}_{j \in \Z^2}$ can simply be computed using Green's formula,
\begin{align*}
 \sqrt{8\pi} \ \hat{q}_1(j)
 & = \int_{\Omega_\RR}q(x)e^{-\i j_1x_1 - \i \frac{j_2\pi}{2} x_2} \d{x} \\
 & = 2 \int_{D} e^{-\i j_1x_1 - \i \frac{j_2\pi}{2} x_2} \d{x}
 = \frac{4\i}{j_2\pi}\int_{\partial D}\nu_2(x)e^{-\i j_1x_1 - i \frac{j_2\pi}{2}x_2} \d{s}\\
 &= \frac{4 \i}{j_2\pi} 
 \int_0^{2\pi}e^{-\i j_1z_1(t)-\i \frac{j_2\pi}{2}\sin(t)}(1.5\sin(t)+2\sin(2t)) \d{t}, \quad j_2 \neq 0.
\end{align*}
This integral can now be accurately evaluated numerically 
(we use the fourth-order convergent composite Simpson's rule). 

Similar techniques yield the Fourier coefficients of the contrast 
\[
  q_2 = \begin{cases}
       1 & \text{ in } D_1 := (-\pi/2,\pi/2)\times(0,0.75), \\
       2 & \text{ in } (-\pi,\pi)\times(-0.75,0.75) \setminus \ol{D_1},
     \end{cases}
\]  
that is plotted in Figure~\ref{fig:contrasts}(b).
The Fourier coefficients of $q_2$ can be computed explicitly, 
\begin{align*}
 \sqrt{8\pi} \ \hat{q}_2(j) 
 & = - \int_{D_1}e^{-\i j_1x_1 - \i j_2\pi x_2/2}\d x 
  + 2 \int_D e^{-\i j_1x_1 - \i j_2\pi x_2/2} \d x , \quad j \in \Z^2.
\end{align*}
Both integrals can of course be computed analytically, 
the first one equals for instance 
$4 \i / (\pi j_1 j_2) \, \sin(j_1 \pi / 2)  [1- \exp(-3/8 \, \pi \i j_1) ]$ 
for $j_{1,2} \not = 0$.

The contrast $q_3$ shown in Figure~\ref{fig:contrasts}(c) is defined 
as a smooth function on a sinusoidally shaped strip $D$. In detail,    
\begin{align*}
  D &= \big\{ \left( \begin{smallmatrix} x_1 \\ x_2 \end{smallmatrix} \right) \in \R^2: \ 
                       -\pi<x_1<\pi, \ \sin(2x_1) < 2 x_2 - 1 <  \sin(2x_1)\big\} \quad \text{and } \quad \\
  q_3(x) & = \begin{cases} e^{-x_2}/3 & \text{for } x = \left( \begin{smallmatrix} x_1 \\ x_2 \end{smallmatrix} \right) \in D, \\ 0 & \text{else}. \end{cases} 
 \end{align*}
In this case the Fourier coefficients of the contrast $q$ can be 
computed semi-analytically using Green's formula
\begin{align*}
 \sqrt{8\pi} \ \hat{q}_3(j)
 &= \int_{\Omega_\RR}q(x)e^{-\i j_1x_1 - \i j_2\pi x_2/2} \d{x} = \frac{1}{3} \int_{D}e^{-\i j_1x_1 - (1+\i j_2\pi/2)x_2} \d{x}\\
 &= \frac{-1/3}{1+\i j_2\pi/2}\int_{\partial D}\nu_2(x)e^{-\i j_1x_1 - (1+\i j_2\pi/2)x_2} \d{s}\\
 &= \frac{-1/3}{1+\i j_2\pi/2}\int_0^{2\pi}e^{-\i j_1t -(1+\i j_2 \pi/2)(\sin(2t)/2+1/2)} \d{t} \\
 & \quad + \frac{1/3}{1+\i j_2\pi/2}\int_0^{2\pi}e^{-\i j_1t -(1+\i j_2 \pi/2)(\sin(2t)/2-1/2)} \d{t}.
\end{align*}
Again, we approximate these integrals with the fourth-order 
convergent composite Simpson's rule to get accurate 
approximations for the Fourier coefficients of $q_3$. 

\begin{remark} Of course, 
a similar integration-by-parts trick with respect to $x_1$ would still work 
if $q_3$ depends in a more complicated way on $x_2$. This shows that 
in principle the Fourier coefficients of contrasts that vary smoothly in 
one variable can be computed by approximating one-dimensional integrals 
of smooth functions.
\end{remark}

Finally, we define the contrast function $q_4$ plotted in 
Figure~\ref{fig:contrasts}(d) -- a contrast function that varies smoothly 
in a $2\pi$-periodic rectangle-shaped structure with support $\ol{D}$, 
$D = (-2.5,2.5)\times(-0.75,0.75)$. In detail, 
\[
 q_4(x) = 2\cos(x_1)^2(x_2+0.75)\quad \text{for } x = (x_1,x_2)^\top \in D
\]
and $q_4(x) = 0$ for points outside of $D$. The Fourier coefficients
of $q_4$ can be explicitly computed using integration-by-parts techniques 
we already used above. Omitting technical details, the result is that 
\[
  \hat{q}(j) = \frac{A(j_1)B(j_2)}{\sqrt{8\pi}} \quad \text{for } j = (j_1,j_2)^\top \in \Z^2,
\] 
where
\begin{align*}
A(j_1) &= \begin{cases}
          \frac{\sin(5 j_1)\big[(2\cos(10) + 1)/{j_1} - 8/{j_1^3}\big] - 4\cos(5 j_1)\sin(10)/{j_1^2}}{1-4/{j_1^2}} & j_1 \in \Z\setminus \{0, \pm 2\}, \\
           \sin(20)/4 + \sin(10) + 5  & j_1 = \pm 2,\\
            \sin(10)/2 + 5 & j_1 = 0,
         \end{cases} \\
B(j_2) &= \begin{cases}
        \frac{6\i}{j_2\pi} \exp(-3 \pi \i j_2 /4) - \frac{8\i}{(j_2\pi)^2}\sin(3 \pi j_2 / 4) & j_2 \neq 0, \\
         9/2 & j_2 = 0.
        \end{cases}
\end{align*}

\begin{remark}
  The last example shows that Fourier coefficients of contrasts of the form $q(x) = f_1(x_1) f_2(x_2)$ 
  can be computed (semi-)analytically if $f_{1,2}$ are trigonometric functions, exponentials, or 
  polynomials. The last example features a linear function $f_2(x_2) =  x_2 + 0.75$, however, 
  higher-degree polynomials could be treated as well using additional integrations by parts 
  reducing the polynomial degree.
\end{remark}

Since explicit analytic solutions for plane wave scattering problems 
involving the contrast functions $q_{1,2,3,4}$ are not known, we check 
convergence rates for these structures by computing a reference solution 
for very large discretization parameter $N$.  For all examples below, this 
reference solution is computed for $N = 3072$ using GMRES with a relative 
residual reduction factor of $10^{-8}$. 
The angle of the  incident plane wave is always chosen as $\theta = \pi/4$  
and the wave number always equals $k = \pi/2$. We check the convergence 
rates from Proposition~\ref{th:convergence} by computing  
scattered fields for discretization parameter $N = 2^n$, $n = 4,\dots,9$. As above, 
the GMRES algorithm is stopped when the relative residual reduction factor is less than $10^{-5}$.
Figure~\ref{fig:errorCurves} shows that the convergence order of the method in 
the energy norm $H^1_{\per}$ is in good agreement with the statement 
of Proposition~\ref{th:convergence}. Further, for all test cases, the rates of the error 
measured in $H^{1/2}_{\per}$ and in $L^2$ are in good agreement with the 
statement of Remark~\ref{th:RemarkConvergence}. Computation times and 
the number of iterations of the GMRES algorithm corresponding to the numerical 
experiments illustrated in Figure~\ref{fig:errorCurves} are shown in Table~\ref{tab:times}. 

\begin{figure}[h!!!!t!!!!!b!!!!!]
    \centering
    \begin{tabular}{c c}
       \hspace*{-5mm}  \includegraphics[height = 6.75cm]{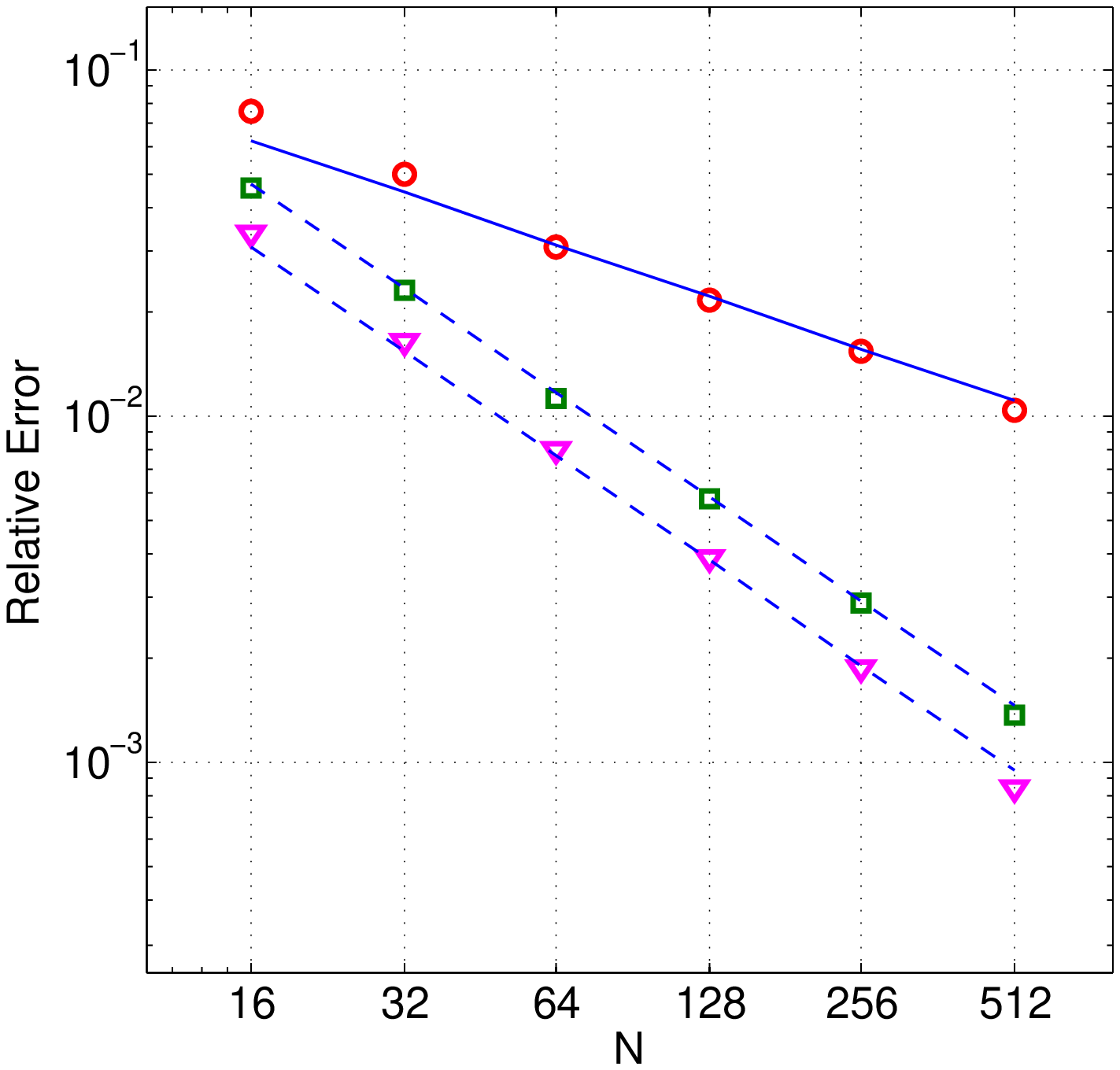}  \hspace*{5mm} & 
         \includegraphics[width = 7cm]{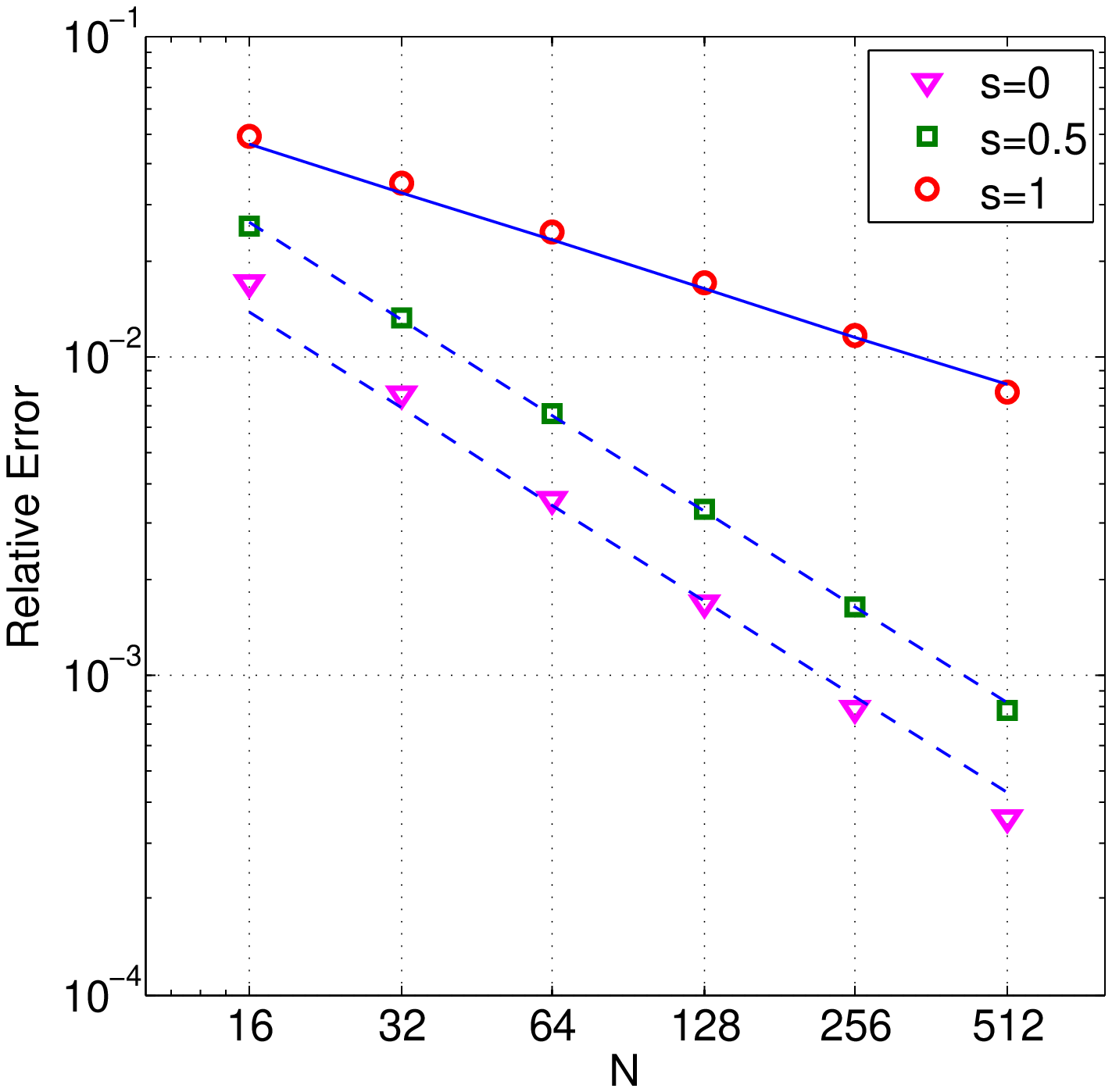}\\ 
       (a) & (b) \\
       \hspace*{-5mm} \includegraphics[width = 7cm]{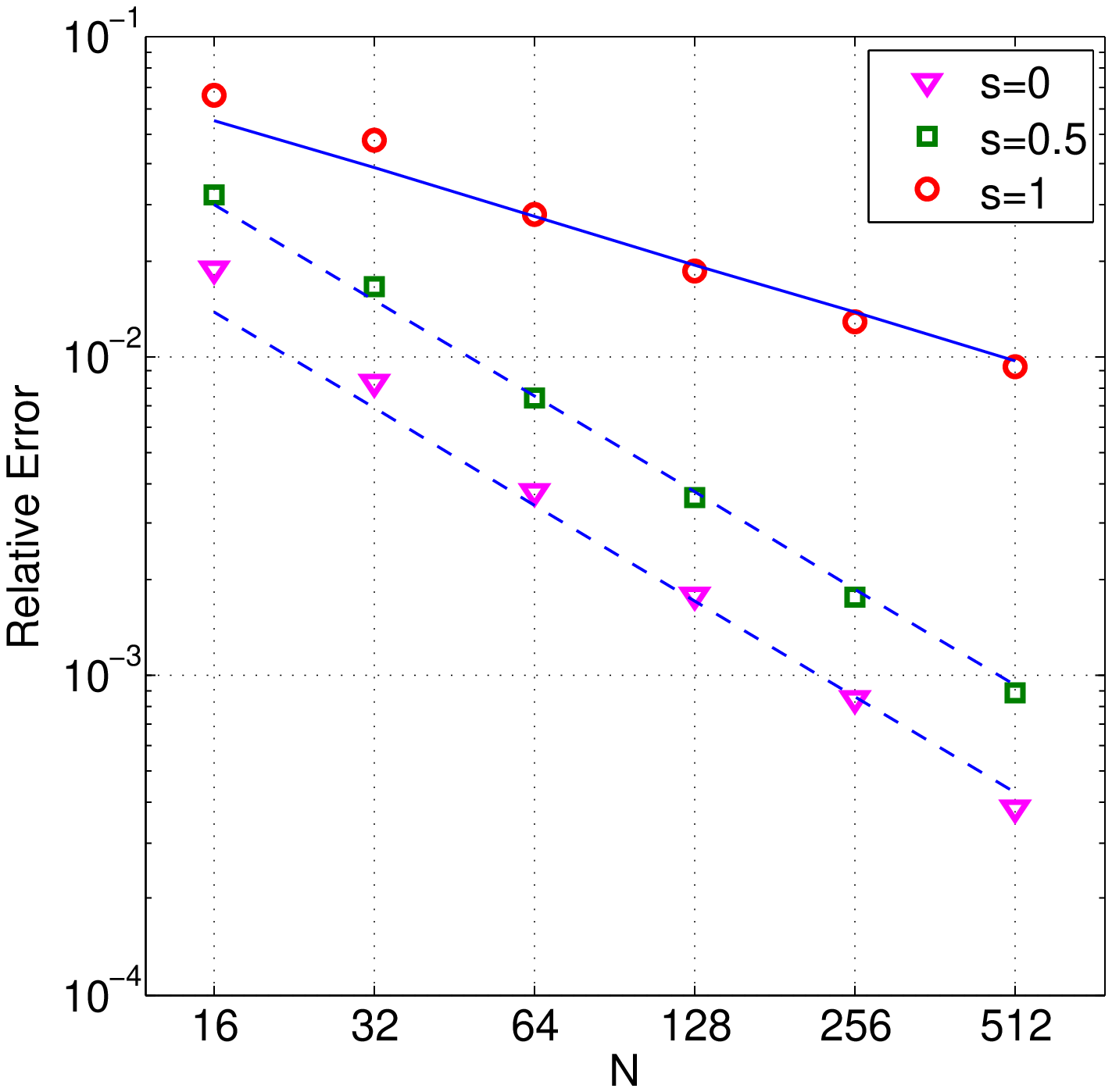}  \hspace*{5mm} & %
          \includegraphics[height = 6.75cm]{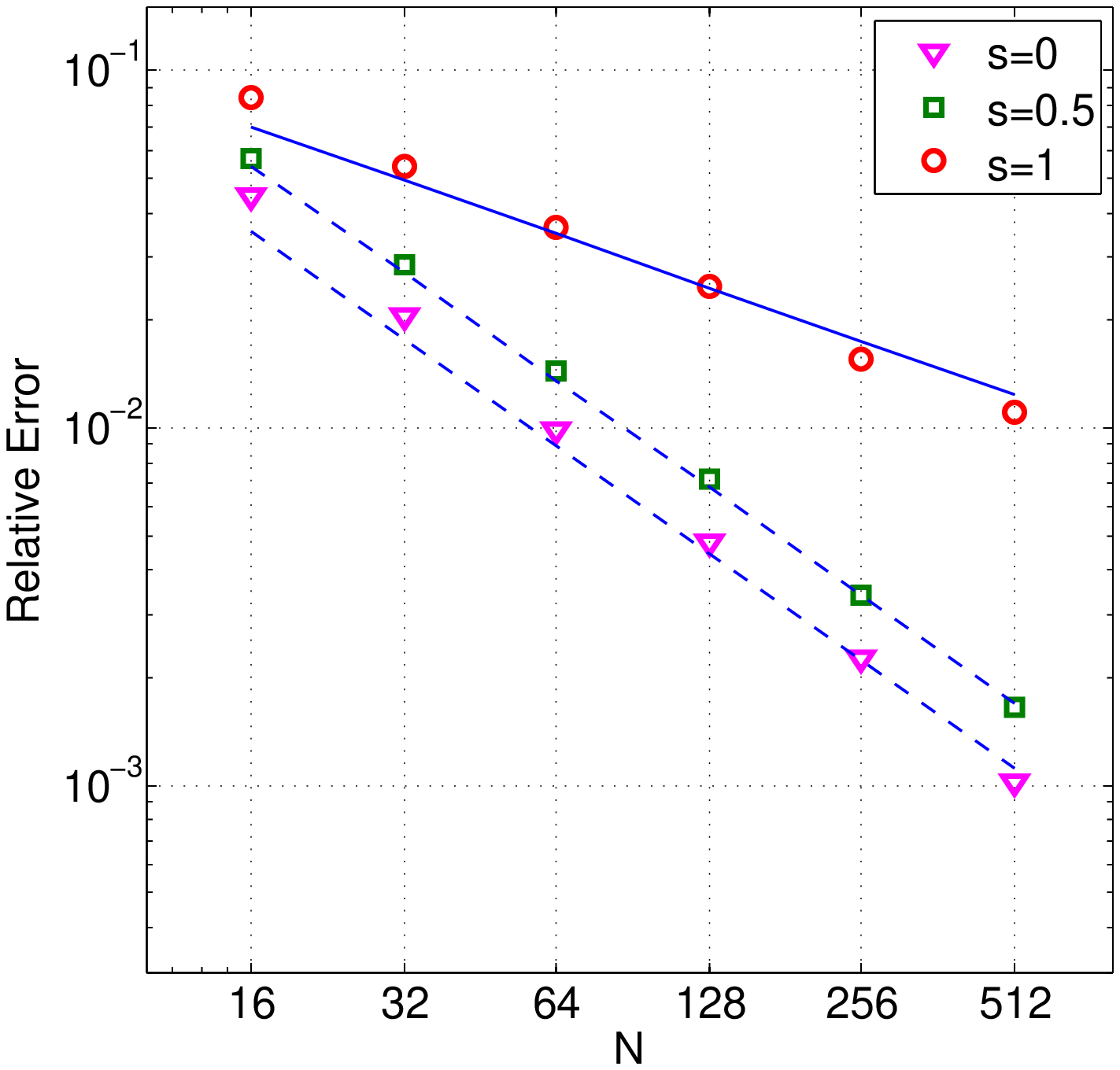} 
         \\ 
        (c) & (d)
    \end{tabular}
        \caption{Test for the convergence rate of the trigonometric Galerkin discretization for the different structures 
        presented in Figure~\ref{fig:contrasts}. The plots show the relative error in the $H^s_{\alpha,\per}$-norm 
        between the approximate solution ($N=2^n$, $n=4,\dots,9$) and the reference solution ($N=3072$), 
        plotted against the discretization parameter $N$. 
        Circles, kites, triangles correspond to $s = 1$, $s = 0.5$ and $s = 0$, respectively. 
        The continuous line and the dotted lines indicate the expected convergence orders 0.5 and 1, 
        respectively.
        (a) Results for the kite-shaped contrast $q_1$ from Figure~\ref{fig:contrasts}(a). 
        (b) Results for the piecewise constant contrast $q_2$ from Figure~\ref{fig:contrasts}(b). 
        (c) Results for the contrast $q_3$ that varies smoothly within a sinusoidal strip from Figure~\ref{fig:contrasts}(c). 
        (d) Results for the contrast $q_4$ that varies smoothly within a rectangle from Figure~\ref{fig:contrasts}(d).}
        \label{fig:errorCurves}
\end{figure}

\begin{table}[h!!!!t!!!b!]
 \centerline{
 \begin{tabular}{ r | c c c c }\hline
 $N$                & 64 & 128 & 256 & 512 \\ \hline  \hline
 Time(s) for $q_1$ (Figure~\ref{fig:contrasts}(a)) & 1.4 & 7 & 44 & 295  \\ 
 Time(s) for $q_2$ (Figure~\ref{fig:contrasts}(b)) & 1.7 & 5 & 16 & 47  \\ 
 Time(s) for $q_3$ (Figure~\ref{fig:contrasts}(c)) & 1.6 & 7 & 39 & 184  \\  
 Time(s) for $q_4$ (Figure~\ref{fig:contrasts}(d)) & 0.4 & 2 & 8 & 39  \\ \hline 
 $\sharp$ GMRES iterations for $q_1$ (Figure~\ref{fig:contrasts}(a)) & 10 & 11 & 11 & 11  \\ 
 $\sharp$ GMRES iterations for $q_2$ (Figure~\ref{fig:contrasts}(b)) & 12 & 12 & 12 & 12  \\ 
 $\sharp$ GMRES iterations for $q_3$ (Figure~\ref{fig:contrasts}(c)) & 6 & 6 & 6 & 6 \\  
 $\sharp$ GMRES iterations for $q_4$ (Figure~\ref{fig:contrasts}(c)) & 9 & 10 & 10 & 10 \\ \hline \hline
 \end{tabular}}
 \caption{Computation times and number of GMRES iterations for the computation of the error curves shown in Figure~\ref{fig:errorCurves}.
 Computing the reference solutions took roughly 1 hour for $q_{2}$ and $q_{4}$, 10 hours for $q_3$ and 16 hours for $q_{1}$.}
 \label{tab:times}
\end{table}

The last computational experiment illustrates the convergence 
of the trigonometric Galerkin technique using an error 
indicator resulting from energy conservation.
Recall the Rayleigh coefficients $\hat{u}^{\pm}_j$ of the scattered field 
from~\eqref{eq:RayleighCondition}. For the incident plane wave $u^i$ with 
incident angle $\theta$, we define similar 
coefficients by $\hat{u}^{i}_j =  \int_{-\pi}^{\pi} u^i(x_1, -h) \exp(-\i\alpha_j x_1) \d{x_1}$
for $j\in \Z$. Then Green's formula applied to equation~\eqref{eq:basic}
together with the Rayleigh expansion condition shows that 
\begin{equation}
\label{eq:energyEq}
 \sum_{j:k^2>\beta_j^2} \beta_j(|\hat{u}^{+}_j|^2 + |\hat{u}^{-}_j + \hat{u}^{i}_j|^2) = \beta_0.
\end{equation}
The sums 
\begin{align*}
E_{\mathrm{tra}}(\theta) := \sum_{j:k^2>\beta_j^2} \beta_j (|\hat{u}^{-}_j + \hat{u}^{i}_j|^2)/\beta_0, \qquad
E_{\mathrm{ref}}(\theta) := \sum_{j:k^2>\beta_j^2} \beta_j|\hat{u}^{+}_j|^2/\beta_0
\end{align*}
correspond to transmitted and reflected wave energies. In the following 
experiment, we compute the function 
\begin{equation}
  \label{eq:conservationError}
  \theta \mapsto |1 - E_{\mathrm{tra}}(\theta) - E_{\mathrm{ref}}(\theta)|
\end{equation} 
for many angles $\theta$ to obtain an error indicator for the numerical accuracy 
of the integral equation solver in dependence on the angle of the incident field. 
This angle, $\theta$, is sampled at 200 points uniformly 
distributed in the interval $[0.2,1.2]$. The wave number $k$ equals 2.5. 
To compute the energy curves shown in Figure~\ref{fig:7}(a) the scattered 
field is approximated in $\mathcal{T}_N$ where $N = 2^8=256$.
The relative residual reduction factor for the GMRES iteration is in this 
experiment always chosen as $10^{-8}$. With this choice, the computation 
time for solving for one fixed incident angle $\theta$ is about 8 seconds.
In Figure~\ref{fig:7}(b) we check the error indicator of energy conservation
from~\eqref{eq:conservationError} for different discretization parameters $N$.
This plot shows that the error of the computed Rayleigh coefficients 
corresponding to propagating modes converges with order 1, exactly as the 
error of the solutions in $H^{1/2}_{\per}$. This seems natural since, 
first, the Rayleigh coefficients are obtained from the numerical solution 
$u_N$ by integration over the line $\Gamma_{\rho} = (-\pi,\pi) \times \{ \rho \}$ 
and, second, the trace theorem states that the mapping 
$u_N \mapsto \left. u_N \right|_{\Gamma_\rho}$ is bounded from 
$H^{s}_{\per}(\Omega_\RR)$ into $L^2(\Gamma_\rho)$ for $s>1/2$. 
The plot in Figure~\ref{fig:7}(b) further shows a slight 
instability around a Wood's anomaly at the angle $\theta = \arccos(1-1/(2.5)) \approx 0.927$, 
as it is going to be expected from Remark~\ref{eq:wood}. (The sampling points 
naturally avoid the exact value of the angle corresponding to this Wood's anomaly.) 

\begin{figure}[!ht]
\centering
\subfloat[]{\includegraphics[width = 7.8cm, height = 7cm]{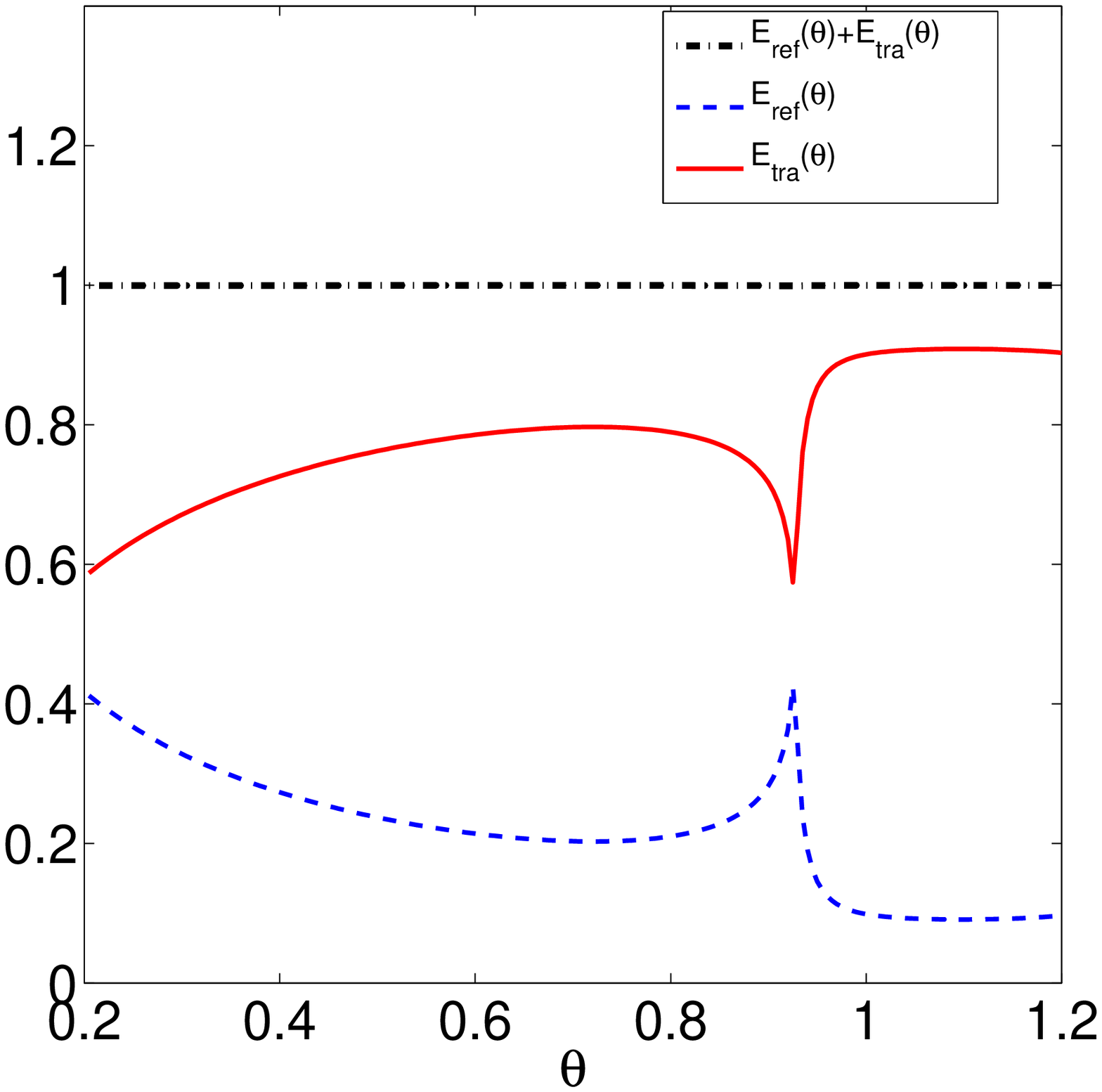}}
\subfloat[]{\includegraphics[width = 7.8cm, height = 7cm]{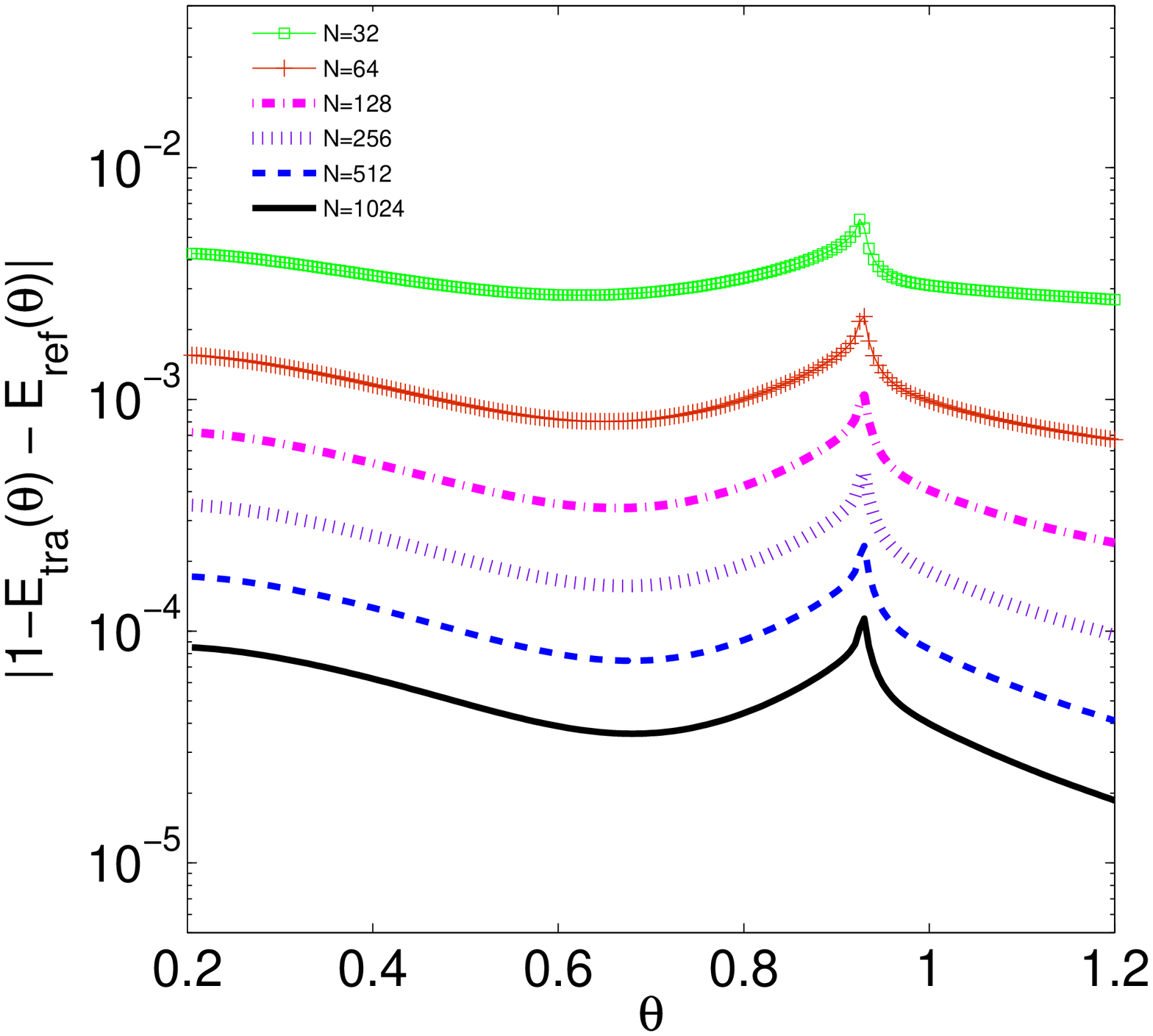}}
\caption{(a) Reflected energy curves (low dashed line) and transmitted 
energy curves (continuous line) plotted against the angle $\theta$ of the 
incident plane wave $u^i$. The two curves sum up to one, as they should due to~\eqref{eq:energyEq}. 
(b) The error criterion~\eqref{eq:conservationError} plotted for different discretization 
parameters $N = 2^n$, $n=5, \dots, 10$, versus the angle $\theta$ of the incident 
plane wave. (The order of the curves from top to bottom corresponds to the 
increasing discretization parameter $N$.)}
\label{fig:7}
\end{figure}


\providecommand{\noopsort}[1]{}
\begin{thebibliography}{30}
\providecommand{\natexlab}[1]{#1}
\providecommand{\url}[1]{\texttt{#1}}
\expandafter\ifx\csname urlstyle\endcsname\relax
  \providecommand{\doi}[1]{doi: #1}\else
  \providecommand{\doi}{doi: \begingroup \urlstyle{rm}\Url}\fi

\bibitem[Arens(2010)]{Arens2006}
Tilo Arens.
\newblock Scattering by biperiodic layered media: The integral equation
  approach, 2010.
\newblock Habilitation Thesis, Universit{\"{a}}t Karlsruhe.

\bibitem[Barnett and Greengard(2011)]{Barne2011}
Alex Barnett and Leslie Greengard.
\newblock A new integral representation for quasi-periodic scattering problems
  in two dimensions.
\newblock \emph{BIT Numerical Mathematics}, 51:\penalty0 67--90, 2011.
\newblock ISSN 0006-3835.
\newblock \doi{10.1007/s10543-010-0297-x}.
\newblock URL \url{http://dx.doi.org/10.1007/s10543-010-0297-x}.

\bibitem[{Bonnet-Ben Dhia} and Starling(1994)]{Bonne1994}
A.-S. {Bonnet-Ben Dhia} and F.~Starling.
\newblock Guided waves by electromagnetic gratings and non-uniqueness examples
  for the diffraction problem.
\newblock \emph{Math. Meth. Appl. Sci.}, 17:\penalty0 305--338, 1994.

\bibitem[Colton and Kress(1992)]{Colto1992a}
David~L. Colton and Rainer Kress.
\newblock \emph{Inverse acoustic and electromagnetic scattering theory}.
\newblock Springer, 1992.
\newblock ISBN 3-540-55518-8, 0-387-55518-8.

\bibitem[Costabel et~al.(2010)Costabel, Darrigrand, and Kon{\'e}]{Costa2010}
M.~Costabel, E.~Darrigrand, and E.H. Kon{\'e}.
\newblock Volume and surface integral equations for electromagnetic scattering
  by a dielectric body.
\newblock \emph{J. Comput. Appl. Math}, 234:\penalty0 1817--1825, 2010.

\bibitem[Costabel et~al.(2011)Costabel, Darrigrand, and Sakly]{Costa2011}
Martin Costabel, Eric Darrigrand, and Hamdi Sakly.
\newblock {The essential spectrum of the volume integral operator in
  electromagnetic scattering by a homogeneous body}.
\newblock November 2011.
\newblock URL \url{http://hal.archives-ouvertes.fr/hal-00646229/en/}.

\bibitem[Elschner and Schmidt(1998)]{Elsch1998}
J.~Elschner and G.~Schmidt.
\newblock Diffraction of periodic structures and optimal design problems of
  binary gratings. {P}art {I}: Direct problems and gradient formulas.
\newblock \emph{Math. Meth. Appl. Sci.}, 21:\penalty0 1297--1342, 1998.

\bibitem[Ewe et~al.(2007)Ewe, Chu, and Li]{Ewe2007}
W.-B. Ewe, H.-S. Chu, and E.-P. Li.
\newblock Volume integral equation analysis of surface plasmon resonance of
  nanoparticles.
\newblock \emph{Opt. Express}, 15:\penalty0 18200--18208, 2007.

\bibitem[Frigo and Johnson(2005)]{Frigo2005}
M.~Frigo and S.G. Johnson.
\newblock The design and implementation of fftw3.
\newblock \emph{Proceedings of the IEEE}, 93\penalty0 (2):\penalty0 216 --231,
  feb. 2005.
\newblock ISSN 0018-9219.
\newblock \doi{10.1109/JPROC.2004.840301}.

\bibitem[Grisvard(1992)]{Grisv1992}
P.~Grisvard.
\newblock \emph{Singularities in Boundary Value Problems}.
\newblock RMA 22. Masson, 1992.

\bibitem[Kelley(1995)]{Kelle1995}
C.~T. Kelley.
\newblock \emph{Iterative Methods for Linear and Nonlinear Equations}.
\newblock Frontiers in Applied Mathematics (No. 16). SIAM, 1995.

\bibitem[Kirsch(1993)]{Kirsc1993}
A.~Kirsch.
\newblock Diffraction by periodic structures.
\newblock In L.~P{\"a}varinta and E.~Somersalo, editors, \emph{Proc. Lapland
  Conf. on Inverse Problems}, pages 87--102. Springer, 1993.

\bibitem[Kirsch and Lechleiter(2009)]{Kirsc2009}
A.~Kirsch and A.~Lechleiter.
\newblock {The operator equations of Lippmann--Schwinger type for acoustic and
  electromagnetic scattering problems in $L^2$}.
\newblock \emph{Applicable Analysis}, 88\penalty0 (6):\penalty0 807--830, 2009.

\bibitem[Kon{\'e}(2010)]{Kone2010}
El-Hadji Kon{\'e}.
\newblock \emph{Equations int{\'e}grales volumiques pour la diffraction d'ondes
  {\'e}lectromagn{\'e}tiques par un corps di{\'e}lectrique}.
\newblock PhD thesis, Universit{\'e} de Rennes I, 2010.
\newblock URL
  \url{http://tel.archives-ouvertes.fr/docs/00/50/49/39/PDF/PhDScript_ElHadji.%
pdf}.

\bibitem[Kottmann and Martin(2000)]{Kottm2000}
J.P. Kottmann and O.J.F. Martin.
\newblock Accurate solution of the volume integral equation for
  high-permittivity scatterers.
\newblock \emph{IEEE Trans. Antennas Propag.}, 48\penalty0 (11):\penalty0
  1719--1726, nov 2000.

\bibitem[Lechleiter and Nguyen(2012)]{Lechl2012a}
Armin Lechleiter and Dinh-Liem Nguyen.
\newblock Volume integral equations for scattering from anisotropic diffraction
  gratings.
\newblock \emph{Mathematical Methods in the Applied Sciences}, pages n/a--n/a,
  2012.
\newblock ISSN 1099-1476.
\newblock \doi{10.1002/mma.2585}.
\newblock URL \url{http://dx.doi.org/10.1002/mma.2585}.

\bibitem[Linton(1998)]{Linto1998}
C.~M. Linton.
\newblock The {G}reen's function for the two-dimensional {H}elmholtz equation
  in periodic domains.
\newblock \emph{J. Eng. Math.}, 33:\penalty0 377--402, 1998.

\bibitem[McLean(2000)]{McLea2000}
W.~McLean.
\newblock \emph{Strongly Elliptic Systems and Boundary Integral Operators}.
\newblock Cambridge University Press, Cambridge, UK, 2000.

\bibitem[N\'{e}d\'{e}lec(2001)]{Nedel2001}
J.-C. N\'{e}d\'{e}lec.
\newblock \emph{Acoustic and Electromagnetic Equations}.
\newblock Springer, New York etc, 2001.

\bibitem[Nie et~al.(2005)Nie, Li, Yuan, Yeo, and Gan]{Nie2005}
Xiao-Chun Nie, Le-Wei Li, Ning Yuan, Tat~Soon Yeo, and Yeow-Beng Gan.
\newblock Precorrected-fft solution of the volume integral equation for 3-d
  inhomogeneous dielectric objects.
\newblock \emph{Antennas and Propagation, IEEE Transactions on}, 53\penalty0
  (1):\penalty0 313 -- 320, jan. 2005.
\newblock ISSN 0018-926X.
\newblock \doi{10.1109/TAP.2004.838803}.

\bibitem[Otani and Nishimura(2009)]{Otani2009}
Y.~Otani and N.~Nishimura.
\newblock An {F}{M}{M} for orthotropic periodic boundary value problems for
  {M}axwell's equations.
\newblock \emph{Waves in Random and Complex Media}, 19:\penalty0 80--104, 2009.
\newblock URL \url{http://dx.doi.org/10.1080/17455030802616863}.

\bibitem[Potthast(1999)]{Potth1999}
R.~Potthast.
\newblock Electromagnetic scattering from an orthotropic medium.
\newblock \emph{J. Int. Eq and Appl.}, 11:\penalty0 179--215, 1999.

\bibitem[Rahola(1996)]{Rahol1996}
J.~Rahola.
\newblock Solution of dense systems of linear equations in the discrete-dipole
  approximation.
\newblock \emph{SIAM J. Sci. Comput.}, 17:\penalty0 78--89, 1996.

\bibitem[Richmond(1965)]{Richm1965}
J.~Richmond.
\newblock Scattering by a dielectric cylinder of arbitrary cross section shape.
\newblock \emph{IEEE Trans. Antennas Propag.}, 13\penalty0 (3):\penalty0
  334--341, 1965.

\bibitem[Richmond(1966)]{Richm1966}
J.~Richmond.
\newblock {TE}-wave scattering by a dielectric cylinder of arbitrary
  cross-section shape.
\newblock \emph{IEEE Trans. Antennas Propag.}, 14\penalty0 (4):\penalty0
  460--464, 1966.

\bibitem[Saranen and Vainikko(2002)]{Saran2002}
J.~Saranen and G.~Vainikko.
\newblock \emph{Periodic integral and pseudodifferential equations with
  numerical approximation}.
\newblock Springer, 2002.

\bibitem[Sauter and Schwab(2007)]{Saute2007}
S.~Sauter and C.~Schwab.
\newblock \emph{Boundary Element Methods}.
\newblock Springer, 1. edition, 2007.
\newblock ISBN 978-3540680925.

\bibitem[Vainikko(2000)]{Vaini2000}
G.~Vainikko.
\newblock Fast solvers of the {L}ippmann-{S}chwinger equation.
\newblock In D.E. Newark, editor, \emph{Direct and Inverse Problems of
  Mathematical Physics}, Int. Soc. Anal. Appl. Comput. 5, page 423, Dordrecht,
  2000. Kluwer.

\bibitem[Zhang and Liu(2002)]{Zhang2002}
Zhong~Qing Zhang and Qing~Huo Liu.
\newblock A volume adaptive integral method ({V}{A}{I}{M}) for 3-{D}
  inhomogeneous objects.
\newblock \emph{Antennas and Wireless Propagation Letters, IEEE}, 1\penalty0
  (1):\penalty0 102 --105, 2002.
\newblock ISSN 1536-1225.
\newblock \doi{10.1109/LAWP.2002.805126}.

\bibitem[Zwamborn and van~den Berg(1992)]{Zwamb1992}
P.~Zwamborn and P.M. van~den Berg.
\newblock The three dimensional weak form of the conjugate gradient {FFT}
  method for solving scattering problems.
\newblock \emph{IEEE Trans. Microwave Theory Tech.}, 40\penalty0 (9):\penalty0
  1757--1766, 1992.

\end{thebibliography}
\end{document}